\pdfoutput=1
\RequirePackage{ifpdf}
\ifpdf % We are running pdfTeX in pdf mode
\documentclass[pdftex]{sigma}
\else
\documentclass{sigma}
\fi

\numberwithin{equation}{section}

\newtheorem{Theorem}{Theorem}[section]
\newtheorem{Corollary}[Theorem]{Corollary}
\newtheorem{Lemma}[Theorem]{Lemma}
\newtheorem{Proposition}[Theorem]{Proposition}
 { \theoremstyle{definition}
\newtheorem{Definition}[Theorem]{Definition}
\newtheorem{Case}{Case}
\newtheorem{Example}[Theorem]{Example}
\newtheorem{Remark}[Theorem]{Remark} }

\begin{document}
\allowdisplaybreaks

\newcommand{\arXivNumber}{1804.09158}

\renewcommand{\thefootnote}{}

\renewcommand{\PaperNumber}{115}

\FirstPageHeading

\ArticleName{The Smallest Singular Values\\ and Vector-Valued Jack Polynomials\footnote{This paper is a~contribution to the Special Issue on the Representation Theory of the Symmetric Groups and Related Topics. The full collection is available at \href{https://www.emis.de/journals/SIGMA/symmetric-groups-2018.html}{https://www.emis.de/journals/SIGMA/symmetric-groups-2018.html}}}
\ShortArticleName{The Smallest Singular Values and Vector-Valued Jack Polynomials}

\Author{Charles F.~DUNKL}

\AuthorNameForHeading{C.F.~Dunkl}

\Address{Department of Mathematics, University of Virginia,\\ PO Box 400137, Charlottesville VA 22904-4137, USA}
\Email{\href{mailto:cfd5z@virginia.edu}{cfd5z@virginia.edu}}
\URLaddress{\url{http://people.virginia.edu/~cfd5z/}}

\ArticleDates{Received June 15, 2018, in final form October 22, 2018; Published online October 25, 2018}

\Abstract{There is a space of vector-valued nonsymmetric Jack polynomials associated with any irreducible representation of a symmetric group. Singular polynomials for the smallest singular values are constructed in terms of the Jack polynomials. The smallest singular values bound the region of positivity of the bilinear symmetric form for which the Jack polynomials are mutually orthogonal. As background there are some results about general finite reflection groups and singular values in the context of standard modules of the rational Cherednik algebra.}

\Keywords{nonsymmetric Jack polynomials; standard modules; Young tableaux}

\Classification{33C52; 20F55; 05E35; 05E10}

\renewcommand{\thefootnote}{\arabic{footnote}}
\setcounter{footnote}{0}

\section{Introduction}

Suppose $W$ is the finite reflection group generated by the reflections in the reduced root system~$R$. This means~$R$ is a finite set of nonzero vectors in $\mathbb{R}^{N}$ such that \thinspace$u,v\in R$ implies \smash{$\mathbb{R}u\cap R= \{ \pm u \} $} and $v\sigma_{u}\in R$ where $\sigma_{u}$ is the reflection $x\mapsto x-2\frac{\langle x,v\rangle}{\langle v,v\rangle}v$ and $\langle\cdot,\cdot\rangle$ is the standard inner product. This implies $ \langle x\sigma_{v},y\sigma_{v} \rangle = \langle x,y\rangle $ for all $x,y\in\mathbb{R}^{N}$ and the group $W$ generated
by $ \{ \sigma_{v}\colon v\in R \} $ is a~finite group of orthogonal transformations of $\mathbb{R}^{N}$. For a fixed vector $b_{0}$ such that $\langle u,b_{0}\rangle\neq0$ for all $u\in R$ there is the decomposition $R=R_{+}\cup R_{-}$ with $R_{+}:= \{ u\in R\colon \langle u,b_{0}
\rangle>0 \} $. The set~$R_{+}$ serves as index set for the reflections in~$W$. In Section~\ref{ExtP} it is assumed that $\operatorname{span}_{\mathbb{R} }(R) =\mathbb{R}^{N}$, while the other sections concerning the symmetric group use the root system~$A_{N-1}$ whose span is \mbox{$\Big\{ x\in\mathbb{R}^{N}\colon \sum\limits_{i=1}^{N}x_{i}=0\Big\} $}. The group $W$ is represented on the space~$\mathcal{P}$ of polynomials in $x= ( x_{1},\ldots,x_{N} ) $ by $wp(x) =p ( xw ) $ for $w\in W$. Denote $\mathbb{N}_{0}:= \{ 0,1,2,\ldots \} $ and for $\alpha\in\mathbb{N}_{0}^{N}$ let $\vert \alpha\vert :=\sum\limits_{i=1}^{N}\alpha_{i}$ and $x^{\alpha}:= \prod\limits_{i=1}^{N}x_{i}^{\alpha_{i}}$, a monomial. Then $\mathcal{P}:=\operatorname{span}\big\{
x^{\alpha}\colon\alpha\in\mathbb{N}_{0}^{N}\big\} $ and \mbox{$\mathcal{P}_{n}:=\operatorname{span}\big\{ x^{\alpha}\colon \alpha
\in\mathbb{N}_{0}^{N}, \,\vert \alpha \vert =n\big\} $} the space of polynomials homogeneous of degree~$n$. Let~$\kappa$ be a parameter (called
multiplicity function), a function on $R$ constant on $W$-orbits. For indecomposable groups $W$ there are at most two orbits in $R$ (two for types $B_{N}$, $F_{4}$ and $I(2k) $, otherwise one for type $A_{N}$, $D_{N}$, $E_{m}$, $I(2k+1)$) then the Dunkl operators $\{\mathcal{D}_{i}\colon 1\leq i\leq N\} $ are defined by
\begin{gather*}
\mathcal{D}_{i}p(x) =\frac{\partial p(x) }{\partial x_{i}}+\sum_{v\in R_{+}}\kappa(v) \frac{p(x) -p( x\sigma_{v}) }{\langle x,v\rangle}v_{i}.
\end{gather*}
Then $\mathcal{D}_{i}\mathcal{D}_{j}=\mathcal{D}_{j}\mathcal{D}_{i}$ for $1\leq i,j\leq N$, $\mathcal{D}_{i}$ maps $\mathcal{P}_{n}$ to $\mathcal{P}_{n-1}$, and the Laplacian $\Delta_{\kappa}:=\sum\limits_{i=1}^{N}\mathcal{D}_{i}^{2}$ satisfies
\begin{gather*}
\Delta_{\kappa}f(x) =\Delta f(x) +\sum_{v\in R_{+}}\kappa(v) \left( \frac{2\langle\nabla f(x) ,v\rangle}{\langle x,v\rangle}- \vert v \vert ^{2}\frac{f (x) -f( x\sigma_{v}) }{\langle x,v\rangle^{2}}\right).
\end{gather*}
The abstract algebra generated by $W$, $\mathcal{D}_{i}$, and multiplication by $x_{i}$, $1\leq i\leq N$, acting on~$\mathcal{P}$ is the rational Cherednik algebra. There are two $W$-invariant bilinear symmetric forms of interest here, denoted $\langle\cdot,\cdot\rangle_{\kappa}$ and $\langle\cdot ,\cdot\rangle_{\kappa,G}$. The first one (called the \textit{contravariant} form) satisfies $\langle\mathcal{D}_{i}f,g\rangle_{\kappa}= \langle f,x_{i}g\rangle_{\kappa}$ for all~$i$ and $f,g\in\mathcal{P}$ and $\langle f,g\rangle_{\kappa}=0$ if $f,g$ are homogeneous of different degrees; also $\langle 1,1\rangle _{\kappa}=1,\langle wf,wg\rangle_{\kappa }=\langle f,g\rangle_{\kappa}$ for $w\in W$. The Gaussian form is derived from the first one by $\langle f,g\rangle_{\kappa,G}:=\big\langle e^{\Delta_{\kappa}/2}f,e^{\Delta_{\kappa}/2}g\big\rangle_{\kappa}$. This form satisfies $\langle\mathcal{D}_{i}f,g\rangle_{\kappa,G}=\langle f,( x_{i}-\mathcal{D}_{i}) g\rangle_{\kappa,G}$ for all~$i$, and thus multiplication by $x_{i}$ is self-adjoint because $\langle f,x_{i}g\rangle_{\kappa,G}=\langle\mathcal{D}_{i}f,g\rangle_{\kappa,G}+\langle f,\mathcal{D}_{i}g\rangle_{\kappa,G}$. For certain constant values of $\kappa$ the Gaussian form is realized as an integral with respect to a~finite positive measure on $\mathbb{R}^{N}$, in fact
\begin{gather*}
\langle f,g\rangle_{\kappa,G}=c_{\kappa}\int_{\mathbb{R}^{N}}f(x) g(x) \prod_{v\in R_{+}}\vert \langle x,v\rangle \vert ^{2\kappa(v) }e^{-\vert x\vert^{2}/2}{\rm d}m_{N}(x),
\end{gather*}
where $m_{N}$ is Lebesgue measure on $\mathbb{R}^{N}$ (see \cite[Theorem~3.10]{Dunkl1991}). The constant $c_{\kappa}$ is a normalizing constant to match $\langle1,1\rangle_{\kappa,G}=1$. The explanation of the value of~$c_{\kappa}$ is in terms of the \textit{fundamental} degrees of~$W$. By a~theorem of Chevalley the ring of $W$-invariant polynomials is generated by~$N$ algebraically independent homogeneous polynomials of degrees $d_{1}\leq d_{2}\leq\cdots\leq d_{N}$ (generally `$<$' holds), and these are the fundamental degrees (see \cite[Section~3.5]{Humphreys1990}). They satisfy $\prod\limits_{i=1}^{N}d_{i}=\#W$ and $\sum\limits_{i=1}^{N}( d_{i}-1) =\#R_{+}$. The Macdonald--Mehta--Selberg integral formula is
\begin{gather*}
\int_{\mathbb{R}^{N}}\prod_{v\in R_{+}} \vert \langle x,v\rangle \vert ^{2\kappa}e^{-\vert x\vert ^{2}/2}{\rm d}m_{N}(x) =c\prod_{i=1}^{N}\frac{\Gamma( 1+d_{i}\kappa ) }{\Gamma ( 1+\kappa) } ,
\end{gather*}
where $c$ is independent of $\kappa$. There is a version of this for the $B_{N}$ and $F_{4}$ types. Etingof \cite[Theorem~3.1]{Etingof2010} gave a~proof of the formula valid for all finite reflection groups. The integral shows that the measure is finite and positive for $\kappa>-\frac{1}{d_{N}}$. \textit{Henceforth we consider only the one-parameter situation with~$W$ having just one conjugacy class of reflections.}

This number $-\frac{1}{d_{N}}$ appears in another context. Suppose for some specific rational value of~$\kappa$ there exists a nonconstant polynomial~$p$ for which $\mathcal{D}_{i}p=0$ for $1\leq i\leq N$, then $p$ is called a~\textit{singular polynomial} and $\kappa$ is a~\textit{singular value}. We can assume that $p$ is homogeneous. In this case $\langle x^{\alpha},p(x) \rangle_{\kappa}=\Big\langle1,\prod\limits_{i=1}^{N}\mathcal{D}_{i}^{\alpha _{i}}p(x) \Big\rangle_{\kappa}=0$ for all $\alpha\in\mathbb{N}
_{0}^{N}$ with $\alpha\neq ( 0,\ldots,0 ) $ and thus $\langle f,p\rangle_{\kappa}=0$ for all~$f\in\mathcal{P}$. Furthermore $\Delta_{\kappa
}p=0$ implying $e^{\Delta_{\kappa}/2}p=p$ and $\langle p,p\rangle_{\kappa,G}=0$. It follows that $\kappa\leq-\frac{1}{d_{N}}$ (taking $\kappa$
constant). In fact the smallest (in absolute value) singular value is indeed~$-\frac{1}{d_{N}}$ \cite[Theorem~4.9]{Dunkl/Jeu/Opdam1994}. The theory can be extended to polynomials taking values in modules of~$W$. Suppose $\tau$ is an irreducible orthogonal representation of $W$ on a (finite-dimensional) real vector space $V$ with basis $ \{ u_{i}\colon 1\leq i\leq\dim V \} $. The space $\mathcal{P}_{\tau}:=\mathcal{P}\otimes V$ has the basis $ \big\{ x^{\alpha}\otimes u_{i}\colon \alpha\in\mathbb{N}_{0}^{N},\,1\leq i\leq\dim V \big\}$. There is a representation of~$W$ on~$\mathcal{P}_{\tau}$ defined to be the linear extension of
\begin{gather*}
w\mapsto w(p(x) \otimes u) :=p(xw)\otimes( \tau(w) u) ,\qquad p\in\mathcal{P},\qquad u\in V.
\end{gather*}
The associated Dunkl operators are the linear extensions of
\begin{gather*}
\mathcal{D}_{i}(p(x) \otimes u) =\frac{\partial p(x) }{\partial x_{i}}\otimes u+\kappa\sum_{v\in R_{+}} \frac{p(x) -p ( x\sigma_{v} ) }{\langle x,v\rangle }v_{i}\otimes ( \tau ( \sigma_{v} ) u ) .
\end{gather*}
The first bilinear form is a modification of the scalar one: let $ \langle \cdot,\cdot \rangle _{V}$ be a $W$-invariant inner product on~$V$, that is, $ \langle \tau(w) u_{1},\tau(w) u_{2}\rangle _{V}=\langle u_{1},u_{2}\rangle _{V}$ for all $u_{1},u_{2}\in V,w\in W$; this form is unique up to multiplication by a~constant. The symmetric form $\langle \cdot,\cdot\rangle _{\kappa}$ satisfies (i)~$\langle 1\otimes u_{1},1\otimes u_{2}\rangle_{\kappa}= \langle u_{1},u_{2} \rangle _{V}$ for $u_{1},u_{2}\in V$, (ii)~$\langle wf,wg\rangle _{\kappa}=\langle f,g\rangle_{\kappa}$ for $f,g\in\mathcal{P}_{\tau}$ and $w\in W$, (iii)~if $f,g\in\mathcal{P}_{\tau}$ are homogeneous of different degrees then
$\langle f,g\rangle _{\kappa}=0$, (iv) $\langle\mathcal{D}_{i}f,g\rangle_{\kappa}=\langle f,x_{i}g\rangle_{\kappa}$ for all $i$. As a~consequence suppose $\alpha\in\mathbb{N}_{0}^{N}$, $\vert \alpha\vert =n$, $u\in V$ and $f\in\mathcal{P}_{\tau}$ is homogeneous of degree~$n$, then $f_{0}:=\prod\limits_{i=1}^{N}\mathcal{D}_{i}^{\alpha_{i}}f ( x) \in V$ and $\langle x^{\alpha}\otimes u,f\rangle _{\kappa}=\langle u,f_{0} \rangle _{V}$. The Gaussian form is defined by $\langle f,g\rangle_{\kappa,G}:=\big\langle e^{\Delta_{\kappa} /2}f,e^{\Delta_{\kappa}/2}g\big\rangle_{\kappa}$ as in the scalar case, and the definition of singular polynomials is the same ($\mathcal{D}_{i}f=0$ for all~$i$, some specific value of~$\kappa$). The interesting question is for what $\kappa$ is the form $ \langle \cdot,\cdot \rangle _{\kappa}$ positive-definite; this property is equivalent to positivity of the Gaussian form. The property $\langle f,x_{i}g\rangle_{\kappa,G}=\langle x_{i} f,g\rangle_{\kappa,G}$ suggests that this form can be realized as an integral over~$\mathbb{R}^{N}$ with a positive matrix-valued measure. Shelley-Abrahamson~\cite{Shelley-Abrahamson2018} proved there is a small interval for $\kappa$ about zero for which this occurs. The interval is a~subset of the interval for which the form is positive. The containment may be proper but the question of equality is not settled as yet.

It is the purpose of this note to show that the positivity interval is bounded by the smallest singular values, to illustrate the theory by constructing singular polynomials for exterior powers of the reflection representation of any~$W$, and to construct vector-valued Jack polynomials which specialize to singular polynomials for the symmetric groups. In this situation the representation is determined by a~partition~$\tau$ of~$N$ and the smallest singular values are $\pm\frac{1}{h_{\tau}}$ where $h_{\tau}$ is the longest hook-length of the Ferrers diagram of~$\tau$ (see Etingof and Stoica \cite[Section~5]{Etingof/Stoica2009}). The isotype (that is, a partition of~$N$) of these singular polynomials is determined.

There are two ways of finding singular polynomials, either define them directly (as in~\cite{Etingof/Stoica2009}) or describe the nonsymmetric Jack polynomials which become singular when specialized to the appropriate parameter value. Feigin and Silantyev~\cite{Feigin/Silantyev2012} found explicit formulas for all singular polynomials which span a $W$-module isomorphic to the reflection representation of~$W$.

The presentation starts with the result on the positivity of the Gaussian form, then the definition and properties of $\mathcal{P}_{\tau}$, the exterior powers of the reflection representation, the nonsymmetric Jack polynomials, results about the action of $\mathcal{D}_{i}$ and the construction of the singular polynomials. The theory of vector-valued nonsymmetric Jack polynomials, originated by Griffeth~\cite{Griffeth2010}, allows detailed analyses of~$\mathcal{P}_{\tau}$. In fact he constructed these polynomials for any group $G(r,1,N)$, the group of $N\times N$ monomial matrices whose nonzero entries are $r^{\rm th}$ roots of unity. In \cite[Section~5]{Griffeth2018} he determined the unitarity locus associated to the contravariant forms associated to these polynomials. These are regions in the parameter space $\mathbb{R}^{r}$ and the highest-dimensional components can be shown to be the regions of positivity.

\section{Region of positivity of the Gaussian form}

Fix an irreducible representation $\tau$ of $W$. The form $\langle\cdot,\cdot\rangle_{\kappa}$ is normalized by $\langle1 \otimes u_{1} ,1 \otimes u_{2} \rangle_{\kappa} = \langle u_{1} ,u_{2} \rangle_{V}$ where $\langle
\cdot, \cdot\rangle_{V}$ is a $W$-invariant bilinear positive symmetric form on $V$ (it is unique up to a multiplicative constant).

\begin{Definition}Let $\Omega$ denote the region of $\kappa\in\mathbb{R}$ for which $\langle f,f\rangle_{\kappa}\geq0$ for all $f\in\mathcal{P}_{\tau}$.
\end{Definition}

The following is due to Shelley-Abrahamson~\cite{Shelley-Abrahamson2018}.

\begin{Theorem}The region $\Omega$ contains a neighborhood of $\kappa=0$.
\end{Theorem}

This result includes the existence of a matrix measure on $\mathbb{R}^{N}$ which realizes the Gaussian form.

\begin{Lemma}Suppose for some $\kappa\in\Omega$ there is a polynomial $f \in\mathcal{P}_{\tau}$ such that $f \neq0$ and $\langle f ,f \rangle_{\kappa} =0$ then the space $X_{f}:=\operatorname{span} \{ w f \colon w \in W \} $ can be decomposed as a sum of irreducible $W$-modules and $g \in X_{f}$ implies $\langle g ,p \rangle_{\kappa} =0$ for all $p \in \mathcal{P}_{\tau}$.
\end{Lemma}

\begin{proof} The decomposability is a group-theoretic property. By the $W$-invariance property of $\langle\cdot,\cdot\rangle_{\kappa}$ it follows that $\langle wf,wf\rangle_{\kappa}=0$ for all $w$. The Cauchy--Schwartz inequality for $\langle\cdot,\cdot\rangle_{\kappa}$ is valid because $\kappa\in\Omega$ thus $\vert \langle wf,p\rangle_{\kappa}\vert ^{2}\leq\langle wf,wf\rangle_{\kappa}\langle p,p\rangle_{\kappa}=0$ for all $w\in W$ and
$p\in\mathcal{P}_{\tau}$. In particular $\langle w_{1}f+w_{2}f,p\rangle _{\kappa}=0$ for any $w_{1},w_{2}$ and thus $g\in X_{f}$ implies $\langle g,p\rangle_{\kappa}=0$.
\end{proof}

We will show that the set of singular values is a subset of a set of rational numbers with no accumulation point, that is, there is a minimum nonzero distance between elements.

\begin{Proposition}The eigenvalues of the class $\sum\limits_{v \in R_{ +}}\sigma_{v}$ considered as a~$($central$)$ transformation of the group algebra~$\mathbb{R} W$ are integers in the interval $[ -\# R_{ +} ,\# R_{ +}] $.
\end{Proposition}

\begin{proof}The basic idea is that the solutions of the characteristic equation are algebraic integers. The details are in \cite[p.~194]{Dunkl/Xu2014}.
\end{proof}

Because the right regular representation of $W$ on $\mathbb{R}W$ is a direct sum of all irreducible representations of $W$ the integer property of eigenvalues applies to $\sum\limits_{v\in R_{+}}\rho(\sigma_{v}) $ for any irreducible representation $\rho$. Since $\rho$ is irreducible and $\sum\limits_{v\in R_{+}}\rho(\sigma_{v}) $ is central there is just one eigenvalue, denoted by~$\varepsilon(\rho) $. Denote the set of equivalence classes of irreducible representations of~$W$ by~$\widehat{W}$.

\begin{Proposition}Suppose $f\in\mathcal{P}_{\tau}$ then $\sum\limits_{i=1}^{N}x_{i}\mathcal{D}_{i}f=\sum\limits_{i=1}^{N}x_{i}\frac{\partial f}{\partial x_{i}}+\kappa\Big( \varepsilon(\tau) f-\sum\limits_{v\in R_{+}}\sigma_{v}f\Big)$. If $f$ is singular and homogeneous of degree~$n$ then $\kappa=\frac{n}{\varepsilon(\rho) -\varepsilon ( \tau) }$ where $\rho\in\widehat{W}$.
\end{Proposition}

\begin{proof}Let $p\in\mathcal{P}_{n}$ and $u\in V$ then
\begin{gather*}
\sum_{i=1}^{N}x_{i}\mathcal{D}_{i}(p(x) \otimes u) =\sum_{i=1}^{N}x_{i}\frac{\partial p(x) }{\partial x_{i}}\otimes u+\kappa\sum_{v\in R_{+}}( p(x) -p(x\sigma_{v})) \otimes( \tau(\sigma_{v})u) \\
\hphantom{\sum_{i=1}^{N}x_{i}\mathcal{D}_{i}(p(x) \otimes u)}{} =np(x) \otimes u+\kappa \bigg( \varepsilon (\tau) p(x) \otimes u-\sum_{v\in R_{+}}\sigma_{v}(p(x) \otimes u) \bigg) .
\end{gather*}
The statement $\sum\limits_{v\in R_{+}}p(x) \otimes ( \tau(\sigma_{v}) u) =\varepsilon(\tau) p(x) \otimes u$ follows from the fact that $V$ is the representation space for $\tau$. The relation is extended to all of $\mathcal{P}_{\tau}$ by linearity. That is, if $f\in\mathcal{P}_{n}\otimes V$ then $\sum\limits_{i=1}^{N}x_{i}\mathcal{D}_{i}( f(x)) =nf(x) +\kappa\varepsilon(\tau) f ( x) -\kappa\sum\limits_{v\in R_{+}}\sigma_{v}f(x) $. If $f$ is singular then it must be an eigenfunction of $\sum\limits_{v\in R_{+}}\sigma_{v}$. The space $\operatorname{span} \{ wf\colon w\in W\} $ consists of singular polynomials (same~$\kappa$) and can be decomposed into irreducible $W$-submodules, thus the eigenvalues of $\sum\limits_{v\in R_{+}}\sigma_{v}$ are ele\-ments of the set $\big\{\varepsilon(\rho) \colon \rho\in\widehat{W}\big\}$. Hence there is some $\rho\in\widehat{W}$ such that $\sum\limits_{v\in R_{+}}\sigma_{v}f=\varepsilon(\rho) f$ and $0=nf+\kappa ( \varepsilon(\tau) -\varepsilon(\rho)) f$. (The case $\varepsilon(\tau) =\varepsilon (\rho) $ is impossible.)
\end{proof}

It is possible that $\varepsilon( \rho_{1}) =\varepsilon(\rho_{2} ) $ for some $\rho_{1},\rho_{2}\in\widehat{W}$ with $\rho _{1}\neq\rho_{2}$, but then~$f$ can be decomposed into two components, each being singular (by convolution with the respective characters). The minimum distance between two singular values is bounded below by $\big(\max\limits_\rho\vert \varepsilon( \tau) -\varepsilon( \rho) \vert \big) ^{-1}$. Recall $\{\varepsilon( \rho)\} $ is a set of integers contained in $[ -\#R_{+},\#R_{+}] $.

For each $n\geq1$ restrict the form $\langle\cdot,\cdot\rangle_{\kappa}$ to $\mathcal{P}_{n}\otimes V$. The condition that the form is positive-definite is that the leading principal minors of the Gram matrix are positive (for example use the basis $\big\{ x^{\alpha}\otimes u_{i}\colon \vert \alpha\vert =n,\,1\leq i\leq\dim V\big\} $). The minors are polynomials in~$\kappa$ and are positive in a neighborhood of~$0$. Let~$z_{n}$ denote the
positive zero of any of the minors closest to $0$, that is the form is positive for $0\leq\kappa<z_{n}$ and is positive-semidefinite for $\kappa=z_{n}$ and there exists $f_{n}\in\mathcal{P}_{n}\otimes V$ such that
$f_{n}\neq0$ and $\langle f_{n},f_{n}\rangle_{\kappa}=0$ (which implies $\langle f_{n},g\rangle_{\kappa}=0$ for any $g\in\mathcal{P}_{n}\otimes V$ by the Cauchy--Schwartz inequality). If there are no positive zeros set $z_{n}=\infty$.

(The reason for the following careful argument is to avoid the hypothetical situation $z_{n} =1 +\frac{1}{n}$, $\min z_{n} =1$ and there is no nonzero polynomial $f$ with $\langle f ,f \rangle_{\kappa} =0$ for $\kappa=1$.)

\begin{Lemma}Suppose $z_{n} >z_{n +1}$ then $z_{n +1}$ is a singular value.
\end{Lemma}

\begin{proof}Set $\kappa=z_{n+1}$. By the definition of $z_{n+1}$ and properties of the form
\begin{gather*}
0=\bigg\langle f_{n+1},\sum_{i=1}^{N}x_{i}\mathcal{D}_{i}f_{n+1}\bigg\rangle_{\kappa }=\sum_{i=1}^{N}\langle\mathcal{D}_{i}f_{n+1},\mathcal{D}_{i}f_{n+1}%
\rangle_{\kappa}.
\end{gather*}
By hypothesis the form is positive-definite on $\mathcal{P}_{n}\otimes V$ for $\kappa=z_{n+1}<z_{n}$. Thus $f_{n+1}$ is sin\-gular.
\end{proof}

Define the subsequence $ \{ z_{n_{i}} \} $ by $n_{1}=\min \{n\colon z_{n}<\infty \} $ and $n_{i+1}=\min \{ n\colon n>n_{i},\, z_{n}<z_{n-1} \} $ (essentially the points of decrease of the sequence). If there are no positive eigenvalues then each $z_{n}=\infty$ and the form is positive-definite for $\kappa\geq0$. Now assume there is at least one $z_{n}<\infty$. Each $z_{n}\geq z_{n_{i}}$ for some~$i$.

\begin{Theorem}Let $z_{0} =\min \{ z_{n_{i}} \colon i \geq1 \} $ then $z_{0} =z_{n_{j}}$ for some $j$, the form $\langle\cdot, \cdot\rangle_{\kappa}$ is positive-definite for $0 \leq\kappa<z_{0}$ and $z_{0}$ is a singular value.
\end{Theorem}

\begin{proof}By the lemma the subsequence consists of singular values. The spacing of singular values implies there is no accumulation point thus the minimum $z_{0}$ is achieved at one of the values $z_{n_{j}}$. Hence there exists $f \in\mathcal{P}_{n_{j}} \otimes V$ such that $f \neq0$ and $f$ is singular for $\kappa=z_{0}$.
\end{proof}

The same argument can be applied to negative $\kappa$: let $z_{n}^{ \prime}$ be the negative zero closest to $0$ of the leading principal minors of the form restricted to $\mathcal{P}_{n} \otimes V$ so the form is positive-definite for $z_{n}^{ \prime} <\kappa\leq0$; if $z_{n +1}^{ \prime} >z_{n}^{ \prime}$ then $z_{n +1}^{ \prime}$ is a singular value and so is $z_{0}^{ \prime} =\max\{ z_{n}^{ \prime}\} $ (excluding the situation of no negative singular values where $z_{0}^{ \prime} = -\infty$).

To summarize there is an interval $z_{0}^{\prime}<\kappa<z_{0}$ for which $\langle\cdot,\cdot\rangle_{\kappa}$ is positive-definite and $z_{0}$, $z_{0}^{\prime}$ are singular values if finite, respectively.

\section{Exterior powers of the reflection representation}\label{ExtP}

Suppose $W$ has only one conjugacy class of reflections and $\operatorname{span}_{\mathbb{R}}(R) =\mathbb{R}^{N}$. Specialize $\tau$ to the reflection representation of $W$ on $V=\mathbb{R}^{N}$. (The previous two statements imply that $W$ is indecomposable and the reflection representation is irreducible.) Let $\wedge^{m}(V) =V\wedge V\wedge\cdots\wedge V$ ($m$ factors) with $1\leq m\leq N$. We will show that $\mathcal{P}_{1}\otimes\wedge^{m}(V)$ has singular polynomials for $\kappa=\pm\frac{1}{d_{N}}$, where $d_{N}$ is the largest fundamental degree of~$W$ (also see \cite[Corollary~4.2]{Etingof/Stoica2009}), and
$1\leq m<N$. Ciubotaru \cite[Section~5]{Ciubotaru2016} proved a necessary condition for the region of positivity for any $W$ (in fact, also for complex reflection groups) and any irreducible $W$-module $U$, which involves the
decomposition of $U\otimes\wedge^{m}(V) $ into $W$-irreducible subspaces.

Let $ \{ u_{i}\colon 1\leq i\leq N \} $ be the standard orthonormal basis of $V$, and let $\mathbf{x}=\sum\limits_{i=1}^{N}x_{i}\otimes u_{i}$.

\begin{Lemma}\label{formabR}For $a,b\in V$ the symmetric bilinear form $\sum\limits_{v\in R_{+}}\frac{\langle a,v\rangle\langle b,v\rangle}{\vert v\vert^{2}}=\frac{\#R_{+}}{N}\langle a,b\rangle$. In particular $\sum\limits_{v\in R_{+}}\frac{v_{i}v_{j}}{\vert v\vert ^{2}}=\frac{\#R_{+}}{N}\delta_{ij}$.
\end{Lemma}

\begin{proof}For $a,b\in V$ define $[a,b] :=\sum\limits_{v\in R_{+}}\frac{1}{[v] ^{2}}\langle a,v\rangle\langle b,v\rangle$ then
\begin{gather*}
 [ a\sigma_{u},b\sigma_{u} ] =\sum\limits_{v\in R_{+}}\frac{1}{\vert v\vert ^{2}}\langle a\sigma_{u},v\rangle\langle b\sigma_{u},v\rangle=\sum\limits_{v\in R_{+}}\frac{1}{\vert v\vert^{2}}\langle a,v\sigma_{u}\rangle\langle b,v\sigma_{u}\rangle=[a,b]
\end{gather*}
for all $u\in R_{+}$. Thus $[ aw,bw] =[a,b] $ for all $w\in W$ and the matrix representing this bilinear form with respect to the basis $ \{ u_{i} \} $ commutes with each~$w$. By hypothesis $\tau$ is irreducible and by Schur's lemma the matrix is a scalar multiple of the identity and $[a,b] =c\langle a,b\rangle$ for all $a,b\in V$. The matrix representing the form $\frac{\langle a,v\rangle\langle b,v\rangle}{\vert v\vert ^{2}}$ has trace $1$ thus the trace for the sum over~$v\in R_{+}$ is~$\#R_{+}$. The form $\langle a,b\rangle$ corresponds to the identity matrix and has trace~$N$. The other conclusion follows from $v_{i}v_{j}=\langle u_{i},v\rangle\langle u_{j},v\rangle$.
\end{proof}

Henceforth assume $\vert v\vert ^{2}=2$ for all $v\in R$, and set $\gamma:=2\frac{\#R_{+}}{N}$, called the \textit{Coxeter number} (see \cite[Section~3.18]{Humphreys1990}); thus $\sum\limits_{v\in R_{+}}v_{i}v_{j}=\gamma\delta_{ij}$. The computations use a boundary operator.

\begin{Definition}Suppose $a ,b_{1} ,b_{2} ,\ldots,b_{m} \in V$ then
\begin{gather*}
\partial(a) ( b_{1} \wedge b_{2} \wedge\cdots\wedge b_{m} ) :=\sum_{i =1}^{m} ( -1 ) ^{i -1} \langle a ,b_{i} \rangle b_{1} \wedge b_{2} \wedge\cdots\wedge \widehat{b_{i}} \wedge\cdots\wedge b_{m},
\end{gather*}
where the caret indicates the omitted factor. The operator $\partial ( a) $ is extended to all of $\wedge^{m}(V) $ by linearity.
\end{Definition}

It can be checked that $\partial(a) $ is well-defined, for example suppose that $b_{m}=\sum\limits_{i=1}^{m-1}c_{i}b_{i}$ then
\begin{gather*}
\partial(a) ( b_{1}\wedge b_{2}\wedge\cdots\wedge b_{m} ) =(-1) ^{m-1}\left( \langle a,b_{m}\rangle -\sum\limits_{i=1}^{m-1}c_{i}\langle a,b_{i}\rangle\right) b_{1}\wedge \cdots\wedge b_{m-1}=0.
\end{gather*}

\begin{Lemma}\label{propDa}Suppose $a ,b_{0} \in V$ and $b \in\wedge^{m}(V) $ then $\partial(a) ( b_{0} \wedge b ) = \langle a ,b_{0} \rangle b -b_{0} \wedge\partial(a) b$ and $\partial (a) ^{2} ( b_{0} \wedge b ) =0$.
\end{Lemma}

\begin{proof}The first part follows directly from the definition. Apply $\partial (a) $ to both sides of the equation
\begin{gather*}
\partial(a) ^{2} ( b_{0} \wedge b) = \langle a ,b_{0} \rangle\partial(a) b -\big\{ \langle a ,b_{0} \rangle\partial(a) b-b_{0} \wedge\partial(a) ^{2} b\big\} =b_{0} \wedge \partial(a) ^{2} b.
\end{gather*} Set $b =b_{1} \wedge\cdots\wedge b_{m}$ and repeatedly use this relation to show $\partial(a) ^{2} b =b_{1} \wedge\cdots\wedge\partial(a) ^{2} b_{m} =0$.
\end{proof}

Denote the exterior power of $\tau$ on $\wedge^{m}(V) $ by~$\tau_{m}$. The operator $\sum\limits_{v\in R_{+}}\tau_{m}( \sigma_{v}) $ acts as multiplication by $\varepsilon( \tau_{m})=\big( \frac{N}{2}-m\big) \gamma$ on $\wedge^{m}(V) $. (Assume that $v=\sqrt{2}u_{1}\in R$ and consider the action of $\tau_{m}(\sigma_{v}) $ on the basis
\begin{gather*}
\{ u_{i_{1}}\wedge u_{i_{2}}\wedge\cdots\wedge u_{i_{m}}\colon 1\leq i_{1}<i_{2}<\cdots<i_{m}\leq N\} ;
\end{gather*}
there are $\binom{N-1}{m-1}$ eigenvectors for the eigenvalue $-1$ and $\binom{N-1}{m}$ eigenvectors for $1$ thus $\operatorname{Tr}( \tau_{m} ( \sigma_{v}) ) $ $=\binom{N-1}{m}-\binom{N-1}{m-1}$, then $\varepsilon ( \tau_{m} ) =\#R_{+}\operatorname{Tr} ( \tau_{m} ( \sigma_{v}) ) /\dim\wedge^{m}(V) =\frac{1}{2}\gamma( N-2m)$.)

\begin{Proposition}Suppose $v \in R$ and $b \in\wedge^{m}(V) $ then $\tau_{m}(\sigma_{v}) b =b -v \wedge\partial(v) b$.
\end{Proposition}

\begin{proof}Let $b =b_{1} \wedge b_{2} \wedge\cdots\wedge b_{m}$. By definition
\begin{gather*}
\tau_{m} (\sigma_{v}) b =( b_{1} \sigma_{v})\wedge\cdots\wedge( b_{m} \sigma_{v}) =( b_{1} - \langle b_{1} ,v \rangle v) \wedge\cdots\wedge( b_{m} - \langle b_{m} ,v\rangle v ) \\
\hphantom{\tau_{m} (\sigma_{v}) b}{} =b - \langle b_{1} ,v \rangle v \wedge b_{2} \wedge\cdots \wedge b_{m} - \langle b_{2} ,v \rangle b_{1} \wedge v \wedge\cdots \wedge b_{m} -\cdots\\
\hphantom{\tau_{m} (\sigma_{v}) b}{} =b -v \wedge( \langle b_{1} ,v \rangle b_{2} \wedge\cdots \wedge b_{m} - \langle b_{2} ,v \rangle b_{1} \wedge b_{3} \wedge\cdots) =b -v\wedge\partial(v) b.\tag*{\qed}
\end{gather*}\renewcommand{\qed}{}
\end{proof}

To compute the terms in $\mathcal{D}_{i} p$ we find $\frac{x_{i} -( x\sigma_{v}) _{i}}{ \langle x ,v \rangle} =v_{i}$ and $\frac{ \langle a,x \rangle- \langle a ,x \sigma_{v} \rangle}{ \langle x ,v \rangle} = \langle
a ,v \rangle$. Note $x$ and $\mathbf{x}$ are different objects, with different transformation rules for~$W$, in fact $\sigma_{v} \mathbf{x} =\mathbf{x}$, because
\begin{gather*}
\sigma_{v} \mathbf{x} =\sum_{i =1}^{N} ( x_{i} - \langle x ,v \rangle v_{i}) \otimes( u_{i} -v_{i} v) =\mathbf{x} -2 \langle x ,v \rangle v + \langle x ,v \rangle\vert v \vert ^{2} v =\mathbf{x}.
\end{gather*}

\begin{Theorem} Suppose $a\in V,b\in\wedge^{m}(V) $ then $\mathbf{x}\wedge b\in\mathcal{P}_{1}\otimes\wedge^{m+1}(V) $, $\sum \limits_{i=1}^{N}a_{i}\mathcal{D}_{i}( \mathbf{x}\wedge b)= ( 1-\gamma\kappa ) a\wedge b$, and $\mathbf{x}\wedge b$ is singular for $\kappa=1/\gamma$.
\end{Theorem}

\begin{proof}Suppose $a\in V$ then
\begin{gather*}
\sum_{i=1}^{N}a_{i}\mathcal{D}_{i}( \mathbf{x}\wedge b) =\sum_{i=1}^{N}a_{i}u_{i}\wedge b+\kappa\sum_{v\in R_{+}}\sum_{j=1}^{N}\langle
a,v\rangle v_{j}\tau_{m}(\sigma_{v}) ( u_{j}\wedge b) \\
\hphantom{\sum_{i=1}^{N}a_{i}\mathcal{D}_{i}( \mathbf{x}\wedge b)}{} =a\wedge b+\kappa\sum_{v\in R_{+}}\langle a,v\rangle\tau_{m}(
\sigma_{v}) ( v\wedge b) \\
\hphantom{\sum_{i=1}^{N}a_{i}\mathcal{D}_{i}( \mathbf{x}\wedge b)}{} =a\wedge b+\kappa\sum_{v\in R_{+}}\langle a,v\rangle ( v\wedge
b-v\wedge\partial(v) ( v\wedge b)) \\
\hphantom{\sum_{i=1}^{N}a_{i}\mathcal{D}_{i}( \mathbf{x}\wedge b)}{} =a\wedge b+\kappa\sum_{v\in R_{+}}\langle a,v\rangle ( v\wedge
b-\langle v,v\rangle v\wedge b+v\wedge v\wedge\partial(v) ) \\
\hphantom{\sum_{i=1}^{N}a_{i}\mathcal{D}_{i}( \mathbf{x}\wedge b)}{}
=a\wedge b-\kappa\sum_{v\in R_{+}}\langle a,v\rangle v\wedge b= (1-\gamma\kappa) a\wedge b,
\end{gather*}
because $\langle v,v\rangle=2$ and $\sum\limits_{v\in R_{+}}\langle a,v\rangle v=\sum\limits_{i,j=1}^{N}\sum\limits_{v\in R_{+}}a_{i}v_{i}v_{j}u_{j} =\gamma\sum\limits_{i=1}^{N}a_{i}u_{i}=\gamma a$.
\end{proof}

\begin{Theorem}Suppose $a\in V$, $b\in\wedge^{m}(V) $ then
\begin{gather*}
\partial ( \mathbf{x}) b\in\mathcal{P}_{1}\otimes\wedge^{m-1}(V) , \qquad \sum\limits_{i=1}^{N}a_{i}\mathcal{D}_{i}\partial ( \mathbf{x}) b= ( 1+\gamma\kappa ) \partial(a) b,
\end{gather*} and $\partial ( \mathbf{x} ) b$ is singular for $\kappa=-1/\gamma$.
\end{Theorem}

\begin{proof}Assume $b=b_{1}\wedge\cdots\wedge b_{m}$ with $b_{1},\ldots,b_{m}\in V$. Then
\begin{gather*}
\sum_{i=1}^{N}a_{i}\mathcal{D}_{i}\partial(\mathbf{x}) b
=\sum_{i=1}^{N}a_{i}\frac{\partial}{\partial x_{i}}\sum_{j=1}^{m} ( -1) ^{j-1}\langle x,b_{j}\rangle\otimes \big( b_{1}\wedge\cdots
\wedge\widehat{b_{j}}\wedge\cdots \big) \\
\hphantom{\sum_{i=1}^{N}a_{i}\mathcal{D}_{i}\partial(\mathbf{x}) b}{} +\kappa\sum_{v\in R_{+}}\langle a,v\rangle\sum_{j=1}^{m}(-1)
^{j-1}\langle v,b_{j}\rangle\tau_{m}(\sigma_{v}) \big(b_{1}\wedge\cdots\wedge\widehat{b_{j}}\wedge\cdots \big) \\
\hphantom{\sum_{i=1}^{N}a_{i}\mathcal{D}_{i}\partial(\mathbf{x}) b}{}
 =\sum_{j=1}^{m}(-1) ^{j-1}\langle a,b_{j}\rangle \otimes \big( b_{1}\wedge\cdots\wedge\widehat{b_{j}}\wedge\cdots \big)
+\kappa\sum_{v\in R_{+}}\langle a,v\rangle\tau_{m}(\sigma_{v}) ( \partial(v) b ) \\
\hphantom{\sum_{i=1}^{N}a_{i}\mathcal{D}_{i}\partial(\mathbf{x}) b}{} =\partial(a) b+\kappa\sum_{v\in R_{+}}\langle a,v\rangle
\big( \partial(v) b-v\wedge\partial(v) ^{2}b\big) = ( 1+\gamma\kappa ) \partial(a)b,
\end{gather*}
because $\partial(v) ^{2}b=0$ (Lemma \ref{propDa}) and
\begin{gather*}
\sum_{v\in R_{+}}\langle a,v\rangle\partial(v) b =\sum_{j=1}^{m}(-1) ^{j-1}\sum_{v\in R_{+}}\langle a,v\rangle\langle
b_{j},v\rangle \big( b_{1}\wedge\cdots\wedge\widehat{b_{j}}\wedge \cdots\big) =\gamma\partial(a) b.\tag*{\qed}
\end{gather*} \renewcommand{\qed}{}
\end{proof}

Here is a table with data on the indecomposable groups with one conjugacy class of reflections. The subscripts indicate the rank $N$ of the group
\begin{gather*}
\left\vert
\begin{matrix}
W & I_{2} ( 2k+1 ) & A_{N} & H_{3} & H_{4} & E_{6} & E_{7} &
E_{8}\\
\#R_{+} & 2k+1 & \frac{N ( N+1 ) }{2} & 15 & 60 & 36 & 63 & 120\\
d_{N} & 2k+1 & N+1 & 10 & 30 & 12 & 18 & 30
\end{matrix}
\right\vert .
\end{gather*}
Thus $\gamma=\frac{2\#R_{+}}{N}=d_{N}$, the largest fundamental degree, also called the Coxeter number (see \cite[Section~3.18]{Humphreys1990}).

Considering the known situation for the symmetric groups and for the reflection representation it appears that for a large collection of representations $\tau$ of degree greater than one that the interval $z_{0}^{\prime}<\kappa<z_{0}$ of positivity is symmetric, $z_{0}^{\prime}=-z_{0}$, and $z_{0}\geq\frac{1}{d_{N}}$. However this is not always the case: there is a nonsymmetric positivity interval arising in two degree~$3$ representations of the icosahedral group~$H_{3}$ (see~\cite{Dunkl2018}).

\section{Representations of the symmetric groups}

The \textit{symmetric group} $\mathcal{S}_{N}$, the set of permutations of $\{ 1 ,2 ,\ldots,N\} $, acts on $\mathbb{C}^{N}$ by permutation of coordinates. The space of polynomials $\mathcal{P}:=\operatorname{span}_{\mathbb{R} (\kappa) } \big\{ x^{\alpha} \colon \alpha\in\mathbb{N}_{0}^{N}\big\} $ where $\kappa$ is a parameter. The action of $\mathcal{S}_{N}$ is extended to polynomials by $w p (x) =p ( x w) $ where $( x w) _{i} =x_{w(i) }$ (consider~$x$ as a row vector and~$w$ as a permutation matrix, $[w] _{i j} =\delta_{i ,w (j) }$, then $x w =x [w] $). This is a representation of $\mathcal{S}_{N}$, that is, $w_{1} ( w_{2} p) (x) =( w_{2} p) ( x w_{1}) =p ( x w_{1} w_{2}) =( w_{1}w_{2}) p (x) $ for all $w_{1} ,w_{2} \in\mathcal{S}_{N}$.

Furthermore $\mathcal{S}_{N}$ is generated by reflections in the mirrors $ \{ x \colon x_{i} =x_{j} \} $ for $1 \leq i <j \leq N$. These are
\textit{transpositions, }denoted by $(i ,j) $, so that $x(i,j) $ denotes the result of interchanging $x_{i}$ and $x_{j}$. Define the $\mathcal{S}_{N}$-action on $\alpha\in\mathbb{Z}^{N}$ so that $( x w) ^{\alpha} =x^{w \alpha}$
\begin{gather*}
(x w) ^{\alpha} =\prod_{i =1}^{N}x_{w (i)}^{\alpha_{i}} =\prod_{j =1}^{N}x_{j}^{\alpha_{w^{ -1} (j) }},
\end{gather*}
that is $(w \alpha) _{i} =\alpha_{w^{ -1} (i) }$.

The \textit{simple reflections} $s_{i}:= ( i ,i +1 ) $, $1 \leq i \leq N -1$, generate~$\mathcal{S}_{N}$. They are the key devices for applying
inductive methods, and satisfy the \textit{braid} relations:
\begin{gather*}
s_{i} s_{j} =s_{j} s_{i} ,\qquad \vert i -j \vert \geq2 ,\qquad s_{i} s_{i +1} s_{i} =s_{i +1} s_{i} s_{i +1}.
\end{gather*}

We consider the situation where the group $\mathcal{S}_{N}$ acts on the range as well as on the domain of the polynomials. We use vector spaces, called $\mathcal{S}_{N}$-modules, on which $\mathcal{S}_{N}$ has an irreducible orthogonal representation $\tau\colon \mathcal{S}_{N} \rightarrow O_{m} ( \mathbb{R}) $ ($\tau(w) ^{ -1} =\tau\big( w^{-1}\big) =\tau(w) ^{T}$). See James and Kerber~\cite{James/Kerber2009} for representation theory, including a modern discussion of Young's methods.

Denote the set of \textit{partitions}
\begin{gather*}
\mathbb{N}_{0}^{N , +}:=\big\{ \lambda\in\mathbb{N}_{0}^{N} \colon \lambda_{1} \geq\lambda_{2} \geq\cdots\geq\lambda_{N}\big\} .
\end{gather*}
We identify $\tau$ with a partition of $N$ given the same label, that is $\tau\in\mathbb{N}_{0}^{N , +}$ and $ \vert \tau \vert =N$. The
length of $\tau$ is $\ell(\tau) :=\max \{ i \colon \tau_{i} >0 \} $. There is a Ferrers diagram of shape $\tau$ (also given the same
label), with boxes at points $(i,j) $ with $1 \leq i \leq \ell(\tau) $ and $1 \leq j \leq\tau_{i}$. A \textit{tableau} of shape $\tau$ is a filling of the boxes with numbers, and a \textit{reverse standard Young tableau} (RSYT) is a filling with the numbers $\{ 1 ,2,\ldots,N\}$ so that the entries decrease in each row and each column.

\begin{Definition}The \textit{hook-length} of the node $(i,j) \in\tau$ is defined to be
\begin{gather*}
h (i,j) :=\tau_{i} -j +\# \{ k \colon i <k \leq\ell ( \tau ) ,\, j \leq\tau_{k} \} +1 ,
\end{gather*}
and the maximum hook-length is $h_{\tau}:=h ( 1 ,1 ) =\tau_{1} +\ell(\tau) -1$.
\end{Definition}

Denote the set of RSYT's of shape $\tau$ by $\mathcal{Y}(\tau) $ and let
\begin{gather*}
V_{\tau}=\operatorname{span}_{\mathbb{R}( \kappa ) } \{ T\colon T\in\mathcal{Y}(\tau) \}
\end{gather*}
with orthogonal basis $\mathcal{Y}(\tau) $. For $1\leq i\leq N$ and $T\in\mathcal{Y}(\tau) $ the entry $i$ is at coordinates $(\operatorname{row}(i,T),\operatorname{col}(i,T)) $ and the \textit{content} is $c(i,T) :=\operatorname{col}(i,T ) -\operatorname{row}(i,T) $. Each $T\in\mathcal{Y}(\tau) $ is uniquely determined by its \textit{content vector} $[c(i,T)] _{i=1}^{N}$. There is an irreducible representation of $\mathcal{S}_{N}$ on~$V_{\tau}$ also denoted by~$\tau$ (slight abuse of notation). To specify the action of $\tau$ it suffices for our purposes to give only the formulae for $\tau(s_{i})$:
\begin{enumerate}\itemsep=0pt
\item[1)] $\operatorname{row}(i,T) =\operatorname{row}(i+1,T) $ (implying $\operatorname{col}(i,T) =\operatorname{col}(i+1,T) +1$ and $c(i,T) -c(i+1,T) =1$) then
\begin{gather*}
\tau ( s_{i} ) T=T;
\end{gather*}
\item[2)] $\operatorname{col}(i,T) =\operatorname{col}(i+1,T) $ (implying $\operatorname{row}(i,T) =\operatorname{row}(i+1,T) +1$ and
$c(i,T) -c(i+1,T) =-1$) then
\begin{gather*}
\tau ( s_{i} ) T=-T;
\end{gather*}
\item[3)] $\operatorname{row}(i,T) < \operatorname{row}(i+1,T) $ and $\operatorname{col}(i,T)>\operatorname{col}(i+1,T) $. In this case
\begin{gather*}
c(i,T) -c(i+1,T) = (\operatorname{col}(i,T) - \operatorname{col}(i+1,T) ) + ( \operatorname{row} (i+1,T) -\operatorname{row}(i,T)) \geq2,
\end{gather*}
and $T^{(i) }$, denoting the tableau obtained from $T$ by exchanging $i$ and $i+1$, is an element of $\mathcal{Y}(\tau) $ and
\begin{gather*}
\tau ( s_{i} ) T=T^{(i) }+\frac{1}{c (i,T) -c(i+1,T) }T,
\end{gather*}
\item[4)] $c(i,T) -c(i+1,T) \leq-2$, thus $\operatorname{row}(i,T) >\operatorname{row}(i+1,T) $ and $\operatorname{col}(i,T)
<\operatorname{col}(i+1,T) $ then with $b=c (i,T) -c(i+1,T) $,
\begin{gather*}
\tau ( s_{i} ) T=\left( 1-\frac{1}{b^{2}}\right) T^{(i) }+\frac{1}{b}T.
\end{gather*}
\end{enumerate}

The formulas in (4) are consequences of those in (3) by interchanging $T$ and $T^{(i) }$ and applying the relations $\tau(s_{i})^{2}=I$ (where $I$ denotes the identity operator on~$V_{\tau}$). The $\mathcal{S}_{N}$-invariant inner product on $V_{\tau}$ is defined by
\begin{gather*}
\langle T,T^{\prime}\rangle_{0}:=\delta_{T,T^{\prime}}\times\prod _{\substack{1\leq i<j\leq N,\\c(i,T) \leq c (j,T) -2}}\left( 1-\frac{1}{( c(i,T) -c ( j,T) ) ^{2}}\right) ,\qquad T,T^{\prime}\in\mathcal{Y}(\tau) .
\end{gather*}
It is unique up to multiplication by a constant.

The \textit{Jucys--Murphy} elements $\omega_{i}:=\sum\limits_{j=i+1}^{N}(i,j) $ satisfy $\sum\limits_{j=i+1}^{N}\tau( ( i,j)) T=c(i,T) T$. Thus $\sum\limits_{1\leq i<j\leq N} \tau ( (i,j) ) $ acts on $V_{\tau}$ as multiplication by
\begin{gather*}
\varepsilon(\tau) =\sum\limits_{j=1}^{N}c(j,T) =\frac{1}{2}\sum\limits_{i=1}^{\ell(\tau) }\tau_{i} (\tau_{i}-2i+1 )
\end{gather*}
(independent of $T\in\mathcal{Y}(\tau)$).

\section{Vector-valued Jack polynomials}

For a given partition $\tau$ of $N$ there is a space of vector-valued nonsymmetric Jack polynomials, also called a standard module of the rational Cherednik algebra. The nonsymmetric vector-valued Jack polynomials (NSJP) form a basis of $\mathcal{P}_{\tau} =\mathcal{P} \otimes\mathcal{V}_{\tau}$, the space of $V_{\tau}$ valued polynomials in~$x$, equipped with the $\mathcal{S}_{N}$ action
\begin{gather*}
w \big( x^{\alpha} \otimes T\big) :=(x w) ^{\alpha} \otimes\tau(w) T ,\qquad \alpha\in\mathbb{N}_{0}^{N} ,\qquad T \in\mathcal{Y} (\tau) ,
\end{gather*}
which is extended by linearity to
\begin{gather*}
w p (x) =\tau(w) p (x w) ,\qquad p \in\mathcal{P}_{\tau}.
\end{gather*}

\begin{Definition}The \textit{Dunkl} and \textit{Cherednik--Dunkl} operators are ($1 \leq i \leq N ,p \in\mathcal{P} ,T \in\mathcal{Y} (\tau) $)
\begin{gather*}
\mathcal{D}_{i} ( p (x) \otimes T ) :=\frac{\partial p (x) }{ \partial x_{i}} \otimes T +\kappa\sum_{j \neq
i}\frac{p (x) -p ( x (i,j) ) }{x_{i} -x_{j}} \otimes\tau ( (i,j) ) T,\\
\mathcal{U}_{i} ( p (x) \otimes T ) :=\mathcal{D}_{i} ( x_{i} p (x) \otimes T ) -\kappa\sum_{j =1}^{i -1}p ( x (i,j) ) \otimes
\tau ( (i,j) ) T,
\end{gather*}
extended by linearity to all of $\mathcal{P}_{\tau}$.
\end{Definition}

The commutation relations analogous to the scalar case hold, that is,
\begin{gather*}
\mathcal{D}_{i} \mathcal{D}_{j} =\mathcal{D}_{j} \mathcal{D}_{i},\qquad \mathcal{U}_{i} \mathcal{U}_{j} =\mathcal{U}_{j} \mathcal{U}_{i} ,\qquad 1 \leq i,j \leq N,\\
w \mathcal{D}_{i} =\mathcal{D}_{w (i) } w , \qquad \forall\, w \in\mathcal{S}_{N}, \qquad s_{j} \mathcal{U}_{i} =\mathcal{U}_{i} s_{j} ,\qquad j \neq i -1 ,i ,\\
s_{i} \mathcal{U}_{i} s_{i} =\mathcal{U}_{i +1} +\kappa s_{i},\qquad \mathcal{U}_{i} s_{i} =s_{i} \mathcal{U}_{i +1} +\kappa, \qquad \mathcal{U}_{i +1} s_{i} =s_{i} \mathcal{U}_{i} -\kappa.
\end{gather*}
The simultaneous eigenfunctions of $ \{ \mathcal{U}_{i} \} $ are called (vector-valued) nonsymmetric Jack polynomials (NSJP). For generic $\kappa$ these eigenfunctions form a basis of $\mathcal{P}_{\tau}$ (\textit{generic} means that $\kappa\neq\frac{m}{n}$ where $m ,n \in \mathbb{Z}$ and $1 \leq n \leq N$). They have a triangularity property with respect to the partial order~$\rhd$ on compositions, which is derived from the dominance order:
\begin{gather*}
\alpha \prec\beta \Longleftrightarrow\sum_{j =1}^{i}\alpha_{j} \leq\sum_{j =1}^{i}\beta_{j} ,\qquad 1 \leq i \leq N ,\qquad \alpha\neq\beta,\\
\alpha\lhd\beta \Longleftrightarrow ( \vert \alpha \vert = \vert \beta \vert ) \wedge\big[ \big( \alpha^{ +}\prec\beta^{ +}\big) \vee\big( \alpha^{ +} =\beta^{ +} \wedge\alpha \prec\beta\big) \big] .
\end{gather*}
There is a subtlety in the leading terms, which relies on the \textit{rank} function:

\begin{Definition}For $\alpha\in\mathbb{N}_{0}^{N}$, $1 \leq i \leq N$
\begin{gather*}
r_{\alpha} (i) =\# \{ j \colon \alpha_{j} >\alpha_{i} \} +\# \{ j \colon 1 \leq j \leq i ,\, \alpha_{j} =\alpha_{i} \} ,
\end{gather*}
then $r_{\alpha} \in\mathcal{S}_{N}$.
\end{Definition}

A consequence is that $r_{\alpha} \alpha=\alpha^{ +}$, the nonincreasing rearrangement of $\alpha$, for any $\alpha\in\mathbb{N}_{0}^{N}$. For example if $\alpha=( 1 ,2 ,1 ,5 ,4) $ then $r_{\alpha} =[ 4 ,3 ,5,1 ,2] $ and $r_{\alpha} \alpha=\alpha^{ +} =( 5 ,4 ,2 ,1,1) $ (recall $w \alpha_{i} =\alpha_{w^{ -1} ( i) }$). Also $r_{\alpha} =I$ if and only if $\alpha$ is a partition ($\alpha_{1} \geq\alpha_{2} \geq\cdots\geq\alpha_{N}$).

\looseness=-1 For each $\alpha\in\mathbb{N}_{0}^{N}$ and $T \in\mathcal{Y} (\tau) $ there is a NSJP $\zeta_{\alpha,T}$ with leading term $x^{\alpha} \otimes\tau\big( r_{\alpha}^{ -1}\big) T$, that~is,
\begin{gather*}
\zeta_{\alpha,T} =x^{\alpha} \otimes\tau\big( r_{\alpha}^{ -1}\big) T +\sum_{\alpha\rhd\beta}x^{\beta} \otimes t_{\alpha\beta} (\kappa) , \qquad t_{\alpha\beta} (\kappa) \in V_{\tau},\\
\mathcal{U}_{i} \zeta_{\alpha,T} = ( \alpha_{i} +1 +\kappa c (r_{\alpha} (i) ,T)) \zeta_{\alpha,T} ,\qquad 1 \leq i \leq N.
\end{gather*}
The list of eigenvalues is called the \textit{spectral vector} $\xi_{\alpha ,T}:=[ \alpha_{i} +1 +\kappa c ( r_{\alpha} (i),T)] _{i =1}^{N}$.

The NSJP's can be constructed by means of a Yang--Baxter graph. The details are in~\cite{Dunkl/Luque2011}; this paper has several figures illustrating some typical graphs.

A node consists of
\begin{gather*}
 ( \alpha,T ,\xi_{\alpha.T} ,r_{\alpha} ,\zeta_{\alpha,T} ),
\end{gather*}
where $\alpha\in\mathbb{N}_{0}^{N}$,$T \in\mathcal{Y} (\tau) ,$ $\xi_{\alpha,T}$ is the spectral vector. The root is
\begin{gather*}
 \big( \mathbf{0} ,T_{0} , [ 1 +\kappa c (i,T_{0})] _{i =1}^{N} ,I ,1 \otimes T_{0}\big),
\end{gather*}
where $T_{0}$ is formed by entering $N ,N -1 ,\ldots,1$ column-by-column in the Ferrers diagram. Proofs by induction in this context typically rely on sequences of applications of $ \{ \tau(s_{i}) \} $ and the inversion number: for $T \in\mathcal{Y} (\tau) $ set
\begin{gather*}
\operatorname{inv} (T) :=\# \{ ( i,j)\colon i <j , \, c (i,T) -c (j,T)\geq2\} .
\end{gather*}
If for particular $T$ and $i$ the relation $c (i,T) -c (i +1 ,T) \geq2$ holds then $T^{(i) }$, the tableau formed by interchanging $i$ and $i+1$ in $T$ is also a RSYT (the relation is equivalent to $\operatorname{row} (i,T) <\operatorname{row}( i +1 ,T) $ and $\operatorname{col}( i,T ) >\operatorname{col}(i +1,T)$; the RSYT property implies that these two inequalities are logically equivalent). In this case
$\operatorname{inv} \big( T^{(i) }\big) =\operatorname{inv} (T) -1$. The inv-maximal tableau is $T_{0}$ and the inv-minimal tableau is $T_{1}$ formed by entering $N ,N -1 ,\ldots,1$ row-by-row.

Steps in the YB-graph correspond to $\tau(s_{i}) $. There are several cases; we start with the situation $\alpha_{i}=\alpha_{i+1}$. These formulae restricted to $1\otimes T$ (so $\alpha=\mathbf{0}$) are equivalent to the definition of the representation $\tau$ on $V_{\tau}$. Throughout the hypotheses are $\alpha\in\mathbb{N}_{0}^{N}$, $T\in\mathcal{Y}(\tau)$, $1\leq i<N$.

\begin{Case}\label{ai=ai1}$\alpha_{i}=\alpha_{i+1}$; define $j:=r_{\alpha}(i) $ implying $r_{\alpha}(i+1) =j+1$
\begin{enumerate}\itemsep=0pt
\item[1)] $\operatorname{row}(j,T) =\operatorname{row}(j+1,T) $ then $s_{i}\zeta_{\alpha,T}=\zeta_{\alpha,T}$;
\item[2)] $\operatorname{col}(j,T) =\operatorname{col}(j+1,T) $ then $s_{i}\zeta _{\alpha,T}=-\zeta_{\alpha,T}$;
\item[3)] $\operatorname{row} (j,T) <\operatorname{row}(j+1,T) $ (thus $c(j,T) -c(j+1,T) \geq2$) set
\begin{gather*}
b^{\prime}:=\frac{1}{c(j,T) -c(j+1,T) }=\frac{\kappa}{\xi_{\alpha,T}(i) -\xi_{\alpha,T}(i+1) },
\end{gather*}
then there is a \textit{step}
\begin{gather*}
( \alpha,T,\xi_{\alpha,T},r_{\alpha},\zeta_{\alpha,T} ) \overset{s_{i}}{\longrightarrow}\big( \alpha,T^{(j) },s_{i}\xi_{\alpha,T},r_{\alpha},\zeta_{\alpha,T^{(j) }}\big),\qquad
\zeta_{\alpha,T^{(j) }}=s_{i}\zeta_{\alpha,T}-b^{\prime}\zeta_{\alpha,T},
\end{gather*}
\end{enumerate}
Note $0<b^{\prime}\leq\frac{1}{2}$ and $\tau ( s_{j} ) T=T^{(j) }+b^{\prime}T$; furthermore the leading term is transformed $s_{i}\big( x^{\alpha}\otimes\tau\big( r_{\alpha}^{-1}\big) T\big) =( xs_{i}) ^{\alpha}\otimes\tau\big( s_{i}r_{\alpha}^{-1}\big) T=x^{\alpha}\otimes\tau\big( r_{\alpha}^{-1}\big) \tau ( s_{j} ) T$ because $s_{i}r_{\alpha}^{-1}=r_{\alpha}^{-1}s_{j}$. The reciprocal relation is
\begin{gather*}
s_{i}\zeta_{\alpha,T^{(j) }}=-b^{\prime}\zeta_{\alpha,T^{(j) }}+\big( 1-b^{\prime2}\big) \zeta_{\alpha,T}.
\end{gather*}
\end{Case}

\begin{Case} $\alpha_{i +1} >\alpha_{i}$, then with
\begin{gather*}
b:=\frac{\kappa}{\xi_{\alpha,T} (i) -\xi_{\alpha,T} ( i +1 ) }
\end{gather*}
there is a \textit{step}
\begin{gather*}
( \alpha,T ,\xi_{\alpha,T} ,r_{\alpha} ,\zeta_{\alpha,T} ) \overset{s_{i}}{ \longrightarrow} ( s_{i} \alpha,T ,s_{i} \xi_{\alpha,T}
,r_{\alpha} s_{i} ,\zeta_{s_{i} \alpha,T} ),\qquad
\zeta_{s_{i} \alpha,T} =s_{i} \zeta_{\alpha,T} -b \zeta_{\alpha,T},
\end{gather*}
and the reciprocal relation is
\begin{gather*}
s_{i} \zeta_{s_{i} \alpha,T} = -b \zeta_{s_{i} \alpha,T} +\big( 1 -b^{2}\big) \zeta_{\alpha,T}.
\end{gather*}
\end{Case}

The reciprocal relations are derived from $s_{i}^{2} =1$. With the aim of letting $\kappa$ take on certain rational values we examine the possible poles in the step rules arising from the factors
\begin{gather*}
\xi_{\alpha,T} (i) -\xi_{\alpha,T} (i +1) =\alpha_{i} -\alpha_{i +1} +\kappa( c ( r_{\alpha} (i) ,T) -c ( r_{\alpha} (i +1) ,T)) .
\end{gather*}
The extreme values of $c (\cdot,T) $ are $\tau_{1} -1$ and $1 -\ell(\tau) $, and thus
\begin{gather*}
 \vert c ( r_{\alpha} (i) ,T ) -c (r_{\alpha} (i +1) ,T ) \vert \leq h_{\tau} -1.
\end{gather*}
Hence $-\frac{1}{h_{\tau} -1} <\kappa<\frac{1}{h_{\tau} -1}$ and $\alpha_{i} \neq\alpha_{i +1}$ imply $\xi_{\alpha,T} (i) -\xi_{\alpha,T}(i +1) \neq0$ (in case $\alpha_{i} =\alpha_{i +1}$ the bound $0 <b^{ \prime} \leq\frac{1}{2}$ applies).

The other links in the YB-graph are degree-raising (affine) operations. Define
\begin{gather*}
\Phi ( a_{1} ,a_{2} ,\ldots,a_{N} ) := ( a_{2} ,a_{3} ,\ldots,a_{N} ,a_{1} +1),\qquad
\theta_{m} :=s_{1} s_{2} \cdots s_{m -1} ,\qquad 2 \leq m \leq N ,
\end{gather*}
so that $\theta_{m}$ is the cyclic permutation $ ( 12 \ldots m ) $. The cycle $\theta_{N}$ interacts with $\Phi$ and the rank function by $r_{\Phi\alpha} =r_{\alpha} \theta_{N}$ (that is, $r_{\alpha} \theta_{N} (i) =r_{\alpha} (i +1) =r_{\Phi\alpha} (i) $ for$~1 \leq i <N$, and $r_{\alpha} \theta_{N} (N) =r_{\alpha} ( 1 ) =r_{\Phi\alpha} (N) $). The \textit{jump} is given by
\begin{gather*}
 ( \alpha,T ,\xi_{\alpha,T} ,r_{\alpha} ,\zeta_{\alpha,T} )
\overset{\Phi}{ \longrightarrow} \big( \Phi\alpha,T ,\Phi\xi_{\alpha,T}
,r_{\alpha} \theta_{N} ,x_{N} \theta_{N}^{ -1} \zeta_{\alpha,T} \big),\qquad
\zeta_{\Phi\alpha,T} =x_{N} \theta_{N}^{ -1} \zeta_{\alpha,T}.
\end{gather*}
The leading term is $x^{\Phi\alpha} \otimes\tau\big( \theta_{N}^{ -1} r_{\alpha}^{ -1}\big) T$ and $\theta_{N}^{ -1} r_{\alpha}^{ -1} = (
r_{\alpha} \theta_{N} ) ^{ -1}$. For example: $\alpha= ( 0 ,2 ,5 ,0) $, $r_{\alpha} =[ 3 ,2 ,1 ,4] $, $\Phi\alpha=( 2,5 ,0 ,1)$, $r_{\Phi\alpha} = [ 2 ,1 ,4 ,3 ] $.

For any $\kappa$ there is a unique bilinear symmetric $\mathcal{S}_{N}$-invariant form on $\mathcal{P}_{\tau}$ which satisfies ($f,g\in \mathcal{P}_{\tau}$):
\begin{gather*}\begin{split}&
\langle1\otimes T,1\otimes T^{\prime}\rangle_{\kappa} =\langle T,T^{\prime }\rangle_{0}, \qquad T,T^{\prime}\in\mathcal{Y}(\tau) ,\\ %\label{admforms}%
&\langle wf,wg\rangle_{\kappa} =\langle f,g\rangle_{\kappa}, \qquad w\in \mathcal{S}_{N}, \qquad \langle\mathcal{D}_{i}f,g\rangle_{\kappa} =\langle f,x_{i}g\rangle_{\kappa}, \qquad 1\leq i\leq N. \end{split}
\end{gather*}
As a consequence $\langle x_{i}f,x_{i}g\rangle_{\kappa}=\langle\mathcal{D}_{i}x_{i}f,g\rangle_{\kappa}=\langle f,\mathcal{D}_{i}x_{i}g\rangle_{\kappa}$ and $\mathcal{U}_{i}=\mathcal{D}_{i}x_{i}-\kappa\sum\limits_{j=1}^{i-1}(i,j) $ is self-adjoint; furthermore $\langle\zeta_{\alpha,T},\zeta_{\beta,T^{\prime}}\rangle_{\kappa}=0$ whenever $( \alpha,T) \neq( \beta,T^{\prime})$ because $\xi_{\alpha,T}\neq\xi_{\beta,T^{\prime}}$ for generic~$\kappa$. The form is defined in terms of $\langle\zeta_{\alpha,T},\zeta_{\alpha,T}\rangle_{\kappa}$ and is extended by linearity and orthogonality to all polynomials. It is a special case of a~result of Griffeth~\cite{Griffeth2010}. The first ingredient is the formula for $\zeta_{\lambda,T}$ for $\lambda\in\mathbb{N}_{0}^{N,+}$ (the Pochhammer symbol is $(t) _{n}=\prod\limits_{i=1}^{n}(t+i-1) $)
\begin{gather}
\langle\zeta_{\lambda,T},\zeta_{\lambda,T}\rangle_{\kappa} =\langle T,T\rangle_{0}\prod_{i=1}^{N} ( 1+\kappa c(i,T) )_{\lambda_{i}}\nonumber\\
\hphantom{\langle\zeta_{\lambda,T},\zeta_{\lambda,T}\rangle_{\kappa}}{} \times\prod_{1\leq i<j\leq N}\prod_{\ell=1}^{\lambda_{i}-\lambda_{j}
}\left( 1-\left( \frac{\kappa}{\ell+\kappa ( c(i,T) -c(j,T) ) }\right) ^{2}\right) .\label{lbnorm}
\end{gather}
The second ingredient expresses the relationship between $\langle\zeta_{\alpha,T},\zeta_{\alpha,T}\rangle_{\kappa}$ and $\langle\zeta_{\alpha^{+}
,T},\zeta_{\alpha^{+},T}\rangle_{\kappa}$. Let
\begin{gather*}
\mathcal{E}(\alpha,T) :=\prod_{\substack{1\leq i<j\leq N\\\alpha_{i}<\alpha_{j}}}\left( 1-\left( \frac{\kappa}{\alpha_{j}
-\alpha_{i}+\kappa ( c ( r_{\alpha}(j) ,T ) -c ( r_{\alpha}(i) ,T ) ) }\right) ^{2}\right) .
\end{gather*}
Then
\begin{gather}
\langle\zeta_{\alpha,T},\zeta_{\alpha,T}\rangle_{\kappa}=\mathcal{E} (\alpha,T ) ^{-1}\langle\zeta_{\alpha^{+},T},\zeta_{\alpha^{+},T}
\rangle_{\kappa},\qquad \alpha\in\mathbb{N}_{0}^{N},\qquad T\in\mathcal{Y} (\tau ) .\label{JPnorm}
\end{gather}

From the bounds on $c (i,T) -c (j,T) $ and the formulae it follows that $\langle\zeta_{\alpha,T} ,\zeta_{\alpha,T} \rangle_{\kappa} >0$ provided $-\frac{1}{h_{\tau}} <\kappa<\frac{1}{h_{\tau}}$. Denote $\langle f ,f \rangle_{\kappa}$ by $ \Vert f \Vert ^{2}$ for any generic value of $\kappa$ (slight abuse of notation).

\section{Differentiation formulae}

First we prove formulae for $\mathcal{D}_{j} \zeta_{\alpha,T}$ for $\ell(\alpha) \leq j \leq N$ (recall the \textit{length of }$\alpha$ is $\ell(\alpha) :=\max \{ i \colon \alpha_{i} >0 \} $). We need the commutation relations (part of the defining relations of the rational Cherednik algebra), $p \in\mathcal{P}_{\tau}$:
\begin{gather}
\mathcal{D}_{i} ( x_{i} p ) -x_{i} \mathcal{D}_{i} p =p+\kappa\sum_{j =1 ,\, j \neq i}^{N}(i,j) p ,\label{DxiDi}\\
\mathcal{D}_{i} ( x_{j} p ) -x_{j} \mathcal{D}_{i} p = -\kappa(i,j) p ,\qquad i \neq j . \label{DixjD}
\end{gather}
Recall the Jucys--Murphy elements $\omega_{i}:=\sum\limits_{j =i +1}^{N}( i ,j) $ for $1 \leq i <N$ and $\omega_{N}:=0$.

\begin{Proposition}Suppose $p \in\mathcal{P}_{\tau}$ and $1 \leq i \leq N$ then $\mathcal{D}_{i} p =0$ if and only if $\mathcal{U}_{i} p =p +\kappa\omega_{i} p$.
\end{Proposition}

\begin{proof}By the above
\begin{gather*}
\mathcal{U}_{i} p =\mathcal{D}_{i} (x_{i} p) -\kappa\sum_{j <i}(i,j) p =x_{i} \mathcal{D}_{i} p +p +\kappa\sum_{j =1 ,j \neq i}^{N}(i,j) p -\kappa\sum_{j <i}(i,j) p\\
\hphantom{\mathcal{U}_{i} p}{} =x_{i} \mathcal{D}_{i} p +p +\kappa\sum_{j =i +1}^{N}(i,j)p.\tag*{\qed}
\end{gather*} \renewcommand{\qed}{}
\end{proof}

\begin{Corollary}\label{dz=zero}Suppose $\alpha\in\mathbb{N}_{0}^{N}$, $T \in\mathcal{Y} (\tau) $ and $\ell(\alpha) <N$ then $\mathcal{D}_{i}
\zeta_{\alpha,T} =0$ for $\ell(\alpha) <i \leq N$.
\end{Corollary}

\begin{proof}By hypothesis $r_{\alpha} (N) =N$ and $\xi_{\alpha} (N) =1 +\kappa c ( N ,T) =1$; also $( 1 +\kappa\omega_{N}) \zeta_{\alpha,T} =\zeta_{\alpha,T}$ thus $\mathcal{U}_{N} \zeta_{\alpha,T} =\zeta_{\alpha,T} p =\zeta_{\alpha,T} +\kappa\omega_{N} \zeta_{\alpha,T}$ and $\mathcal{D}_{N} \zeta_{\alpha,T} =0$. Proceeding by induction suppose $\mathcal{D}_{i} \zeta_{\beta,T^{ \prime}} =0$ for $m \leq i \leq N$ with $m >\ell(\alpha) +1$ and all $T^{ \prime} \in\mathcal{Y} (\tau) $ and $\beta$ with $\ell (\beta) \leq\ell(\alpha)$. Then $\mathcal{D}_{m -1}
\zeta_{\beta,T^{ \prime}} =s_{m -1} \mathcal{D}_{m} s_{m -1} \zeta_{\beta,T^{\prime}}$. By the transformations in Case~\ref{ai=ai1} (note $\beta_{m -1} =\beta_{m} =0$) $s_{m -1} \zeta_{\beta,T^{ \prime}} = \pm\zeta_{\beta,T^{\prime}}$ or $s_{m -1} \zeta_{\beta,T^{ \prime}}$ is a~linear combination (independent of $\kappa$) of $\zeta_{\beta,T^{ \prime}}$ and $\zeta_{\beta,T^{\prime\prime}}$ where $T^{ \prime\prime}$ is the result of interchanging~$m$ and $m -1$ in $T^{ \prime}$. In any of these cases $\mathcal{D}_{m} s_{m -1} \zeta_{\beta,T^{ \prime}} =0$ by the induction hypothesis.
\end{proof}

Suppose $\ell(\alpha) =m<N$. We turn to the evaluation of $\mathcal{D}_{m}\zeta_{\alpha,T}$. Define
\begin{gather*}
\widehat{\alpha}= ( \alpha_{m}-1,\alpha_{1},\alpha_{2},\ldots,\alpha_{m-1},0,\ldots,0 ) .
\end{gather*}
The use of $\widehat{\alpha}$ appeared in Knop \cite{Knop1997} in a creation formula for nonsymmetric Macdonald polynomials. We prove the following in several steps:

\begin{Theorem}\label{Dzetam}Suppose $\alpha\in\mathbb{N}_{0}^{N}$, $T \in\mathcal{Y} (\tau ) $ and $\ell(\alpha) =m <N$ then
\begin{gather*}
\mathcal{D}_{m} \zeta_{\alpha,T} =\frac{ \Vert \zeta_{\alpha,T} \Vert ^{2}}{\Vert \zeta_{\widehat{\alpha} ,T}\Vert ^{2}} \theta_{m}^{ -1}
\zeta_{\widehat{\alpha} ,T}.
\end{gather*}
\end{Theorem}

\begin{Proposition}\label{UDjDU}Suppose $i \neq j$ then
\begin{gather*}
\mathcal{U}_{i} \mathcal{D}_{j} -\mathcal{D}_{j} \mathcal{U}_{i} =\kappa\mathcal{D}_{\min(i,j) } (i,j) .
\end{gather*}
\end{Proposition}

\begin{proof}Suppose $i <j$ then (by use of (\ref{DixjD}))
\begin{gather*}
\mathcal{U}_{i} \mathcal{D}_{j} =\mathcal{D}_{i} x_{i} \mathcal{D}_{j} -\kappa\sum_{s <i}( i ,s) \mathcal{D}_{j}
 =\mathcal{D}_{i} ( \mathcal{D}_{j} x_{i} +\kappa(i,j)) -\kappa\mathcal{D}_{j} \sum_{s <i}(i,s) \\
\hphantom{\mathcal{U}_{i} \mathcal{D}_{j}}{} =\mathcal{D}_{j} \mathcal{U}_{i} +\kappa\mathcal{D}_{i}( i,j) ,
\end{gather*}
because $\mathcal{D}_{i} \mathcal{D}_{j} =\mathcal{D}_{j} \mathcal{D}_{i}$. Suppose $i >j$ then
\begin{gather*}
\mathcal{U}_{i} \mathcal{D}_{j} =\mathcal{D}_{i} x_{i} \mathcal{D}_{j} -\kappa\sum_{s <i ,s \neq j}(i,s) \mathcal{D}_{j}
-\kappa(i,j) \mathcal{D}_{j}\\
\hphantom{\mathcal{U}_{i} \mathcal{D}_{j}}{} =\mathcal{D}_{i} ( \mathcal{D}_{j} x_{i} +\kappa(i,j)) -\kappa\mathcal{D}_{j} \sum_{s <i ,\,s \neq j}(i,s)-\kappa\mathcal{D}_{i} (i,j) =\mathcal{D}_{j} \mathcal{U}_{i} +\kappa\mathcal{D}_{j} ( i,j).\tag*{\qed}
\end{gather*}\renewcommand{\qed}{}
\end{proof}

\begin{Proposition}The spectral vector of $\theta_{m} \mathcal{D}_{m} \zeta_{\alpha,T}$ equals $\xi_{\widehat{\alpha} ,T}$.
\end{Proposition}

\begin{proof}For the rank $r_{\widehat{\alpha}}$ consider
\begin{gather*}
r_{\widehat{\alpha}} (1) =\# \{ j \colon 1 \leq j <m,\, \alpha_{j} >\alpha_{m} -1 \} +1
 =\# \{ j \colon 1 \leq j <m ,\, \alpha_{j} \geq\alpha_{m} \} +1 =r_{\alpha} (m),
\end{gather*}
and for $1 <i \leq m$
\begin{gather*}
r_{\widehat{\alpha}} (i) =\# \{ j \colon 2 \leq j \leq i,\, \alpha_{j -1} \geq\alpha_{i -1} \} +\# \{ j \colon j >i ,\, \alpha_{j -1}
>\alpha_{i -1} \} +c_{i},
\end{gather*}
where $c_{i} =1$ if $\alpha_{m} -1 \geq\alpha_{i -1}$ equivalently if $\alpha_{m} >\alpha_{i -1}$ and $c_{i} =0$ otherwise, thus $r_{\widehat
{\alpha}} (i) =r_{\alpha} ( i -1) $. Apply~(\ref{DxiDi}) with $i =m$ to $\mathcal{D}_{m} \zeta_{\alpha,T}$ to obtain
\begin{gather*}
\mathcal{D}_{m} x_{m} \mathcal{D}_{m} \zeta_{\alpha,T} =\mathcal{D}_{m} \left( \mathcal{D}_{m} x_{m} -1 -\kappa\sum_{i <m}(i,m)
-\kappa\sum_{i >m}(i,m) \right) \zeta_{\alpha,T}\\
\hphantom{\mathcal{D}_{m} x_{m} \mathcal{D}_{m} \zeta_{\alpha,T}}{} =\mathcal{D}_{m} ( \mathcal{U}_{m} -1) \zeta_{\alpha,T}
-\kappa\sum_{i >m}(i,m) \mathcal{D}_{i} \zeta_{\alpha,T} =( \xi_{\alpha,T} (m) -1) \mathcal{D}_{m}\zeta_{\alpha,T},
\end{gather*}
because $\mathcal{D}_{i} \zeta_{\alpha,T} =0$ for $i >m$, and $\theta_{m} \mathcal{D}_{m} x_{m} \theta_{m}^{ -1} =\mathcal{D}_{1} x_{1} =\mathcal{U}_{1}$. Thus $\mathcal{U}_{1} \theta_{m} \mathcal{D}_{m} \zeta_{\alpha,T} = ( \xi_{\alpha,T} (m) -1 ) \theta_{m} \mathcal{D}_{m} \zeta_{\alpha,T}$. Suppose $1 <i \leq m$ then by the commutation relations for $\mathcal{U}_{i}$ and $s_{j}$ we obtain $\theta_{m}^{ -1}
\mathcal{U}_{i} \theta_{m} =\mathcal{U}_{i -1} -\kappa(i -1 ,m)$. Apply this operator to $\mathcal{D}_{m} \zeta_{\alpha,T}$:
\begin{gather*}
\theta_{m}^{ -1} \mathcal{U}_{i} \theta_{m} \mathcal{D}_{m} \zeta_{\alpha,T} =\mathcal{U}_{i -1} \mathcal{D}_{m} \zeta_{\alpha,T} -\kappa( i -1
,m) \mathcal{D}_{m} \zeta_{\alpha,T}\\
\hphantom{\theta_{m}^{ -1} \mathcal{U}_{i} \theta_{m} \mathcal{D}_{m} \zeta_{\alpha,T}}{} =( \mathcal{D}_{m} \mathcal{U}_{i -1} +\kappa\mathcal{D}_{i -1}(i -1 ,m) ) \zeta_{\alpha,T} -\kappa( i -1,m) \mathcal{D}_{m} \zeta_{\alpha,T}\\
\hphantom{\theta_{m}^{ -1} \mathcal{U}_{i} \theta_{m} \mathcal{D}_{m} \zeta_{\alpha,T}}{} =\xi_{\alpha,T} (i -1) \mathcal{D}_{m} \zeta_{\alpha,T}
\end{gather*}
by use of Proposition~\ref{UDjDU} and $(i -1 ,m) \mathcal{D}_{m} =\mathcal{D}_{i -1} (i -1 ,m) $. Suppose $m <i \leq N$ then $\theta_{m}^{ -1} \mathcal{U}_{i} \theta_{m} =\mathcal{U}_{i}$ and
\begin{gather*}
\theta_{m}^{ -1} \mathcal{U}_{i} \theta_{m} \mathcal{D}_{m} \zeta_{\alpha,T} = ( \mathcal{D}_{m} \mathcal{U}_{i} +\kappa\mathcal{D}_{m} ( i,m)) \zeta_{\alpha,T} =\xi_{\alpha,T} (i) \mathcal{D}_{m} \zeta_{\alpha,T}+\kappa(i,m) \mathcal{D}_{i} \zeta_{\alpha,T}\\
\hphantom{\theta_{m}^{ -1} \mathcal{U}_{i} \theta_{m} \mathcal{D}_{m} \zeta_{\alpha,T}}{} =\xi_{\alpha,T} (i) \mathcal{D}_{m} \zeta_{\alpha,T}.
\end{gather*}
Thus the spectral vector of $\theta_{m} \mathcal{D}_{m} \zeta_{\alpha,T}$ is%
\begin{gather*}
\xi_{\widehat{\alpha} ,T} =\big( \xi_{\alpha,T} (m) -1,\xi_{\alpha,T} (1) ,\ldots,\xi_{\alpha,T} ( m -1 ),\xi_{\alpha,T} ( m +1 ), \ldots\big) .\tag*{\qed}
\end{gather*}
\renewcommand{\qed}{}
\end{proof}

We can now finish the proof of Theorem \ref{Dzetam}. If $m =1$ then $\widehat{\alpha} = ( \alpha_{1} -1 ,0, \ldots ) $ and $\theta_{1}=1$.

\begin{proof}From the proposition and the uniqueness of spectral vectors it follows that $\mathcal{D}_{m}\zeta_{\alpha,T}= b\theta_{m}^{-1}\zeta_{\widehat{\alpha},T}$ for some constant $b$. Let $\beta= ( \alpha_{1},\ldots,\alpha _{m-1},0,\ldots,0,\alpha_{m} ) $, thus $\Phi\widehat{\alpha}=\beta$ and $x_{N}\theta_{N}^{-1}\zeta_{\widehat{\alpha},T}$ $=\zeta_{\beta,T}$. Using the properties of the bilinear form we find
\begin{gather*}
\big\langle\mathcal{D}_{m}\zeta_{\alpha,T},\theta_{m}^{-1}\zeta_{\widehat{\alpha },T}\big\rangle =b\big\langle\theta_{m}^{-1}\zeta_{\widehat{\alpha},T},\theta_{m}^{-1}\zeta_{\widehat{\alpha},T}\big\rangle=b \Vert \zeta_{\widehat{\alpha },T}\Vert ^{2}\\
\hphantom{\big\langle\mathcal{D}_{m}\zeta_{\alpha,T},\theta_{m}^{-1}\zeta_{\widehat{\alpha },T}\big\rangle }{} =\big\langle\zeta_{\alpha,T},x_{m}\theta_{m}^{-1}\zeta_{\widehat{\alpha},T}\big\rangle=\big\langle\zeta_{\alpha,T},\big( x_{m}\theta_{m}^{-1}\theta _{N}\big) \theta_{N}^{-1}\zeta_{\widehat{\alpha},T}\big\rangle.
\end{gather*}
Then $x_{m}\theta_{m}^{-1}\theta_{N}=x_{m} ( s_{m+1}s_{m+2}\cdots s_{N-1} ) = ( s_{m+1}s_{m+2}\cdots s_{N-1}\ ) x_{N}$ so that $\big( x_{m}\theta_{m}^{-1}\theta_{N}\big) \theta _{N}^{-1}\zeta_{\widehat{\alpha},T}$ $= ( s_{m+1}s_{m+2}\cdots s_{N-1} ) \zeta_{\beta,T}$. By the transformation rules $ ( s_{m+1}s_{m+2}\cdots s_{N-1} ) \zeta_{\beta,T}=\zeta _{\alpha,T}+\sum_{\gamma,T^{\prime}}b_{\gamma,T^{\prime}}\zeta_{\gamma ,T^{\prime}}$ where $b_{\gamma,T^{\prime}}\in\mathbb{Q}(\kappa)$ and each $\gamma$ satisfies $\gamma_{m}=0$ (and $\gamma_{i}=\alpha_{i}$ for $1\leq i<m$). By the orthogonality of the NSJP it follows that $\big\langle \mathcal{D}_{m}\zeta_{\alpha,T},\theta_{m}^{-1}\zeta_{\widehat{\alpha} ,T}\big\rangle=\Vert \zeta_{\alpha,T}\Vert ^{2}$. This completes the proof.
\end{proof}

We specialize the formula to partition labels.

\begin{Theorem}Suppose $\alpha\in\mathbb{N}_{0}^{N , +}$, $T \in\mathcal{Y} (\tau) $ and $\ell(\alpha) =m <N$ then
\begin{gather*}
\mathcal{D}_{m} \zeta_{\alpha,T} = ( \alpha_{m} +\kappa c ( m,T)) \theta_{m}^{ -1} \zeta_{\widehat{\alpha} ,T}\\
\hphantom{\mathcal{D}_{m} \zeta_{\alpha,T} =}{} \times\prod_{j =m +1}^{N}\frac{ ( \lambda_{m} +\kappa ( c (m ,T ) -c (j,T) -1 ) ) ( \lambda_{m}+\kappa ( c ( m ,T) -c (j,T) +1)) }{( \lambda_{m} +\kappa( c ( m ,T) -c(j,T))) ^{2}}.
\end{gather*}
\end{Theorem}

\begin{proof}The multiplicative constant is
\begin{gather*}
\frac{\Vert \zeta_{\alpha,T}\Vert ^{2}}{\Vert \zeta_{\widehat{\alpha},T}\Vert ^{2}}=\mathcal{E}( \widehat{\alpha},T) \frac{\Vert \zeta_{\alpha,T}\Vert ^{2}}{\Vert\zeta_{\widehat{\alpha}^{+},T}\Vert ^{2}}.
\end{gather*}
Note $\widehat{\alpha}^{+}=( \alpha_{1},\ldots,\alpha_{m-1},\alpha_{m}-1,0,\ldots) $ and
\begin{gather*}
\left\{ \prod_{i=1}^{N} ( 1+\kappa c(i,T) ) _{\alpha_{i}}\right\} \times\left\{ \prod_{i=1}^{N}( 1+\kappa c(i,T)) _{\widehat{\alpha}_{i}^{+}}\right\} ^{-1}=\alpha_{m}+\kappa c(m,T) .
\end{gather*}
After some cancellations in formula (\ref{lbnorm}) we find that
\begin{gather*}
\frac{\Vert \zeta_{\alpha,T}\Vert ^{2}}{\Vert \zeta_{\widehat{\alpha}^{+},T}\Vert ^{2}} =( \alpha_{m}+\kappa c(m,T)) \prod_{i=1}^{m-1}\left( 1-\left( \frac{\kappa}{\alpha_{i}-\alpha_{m}+1+\kappa( c(i,T) -c(m,T)) }\right) ^{2}\right) ^{-1}\\
\hphantom{\frac{\Vert \zeta_{\alpha,T}\Vert ^{2}}{\Vert \zeta_{\widehat{\alpha}^{+},T}\Vert ^{2}} =}{} \times\prod_{j=m+1}^{N}\left( 1-\left( \frac{\kappa}{\alpha_{m}+\kappa ( c(m,T) -c(j,T) ) }\right)^{2}\right) .
\end{gather*}
Also
\begin{gather*}
\mathcal{E} ( \widehat{\alpha},T ) =\prod_{j=2}^{m}\left(1-\left( \frac{\kappa}{\widehat{\alpha}_{j}-\widehat{\alpha}_{1}+\kappa( c( r_{\widehat{\alpha}}(j) ,T)-c( r_{\widehat{\alpha}}(1) ,T)) }\right)^{2}\right) \\
\hphantom{\mathcal{E} ( \widehat{\alpha},T )}{} =\prod_{j=2}^{m}\left( 1-\left( \frac{\kappa}{\alpha_{j-1}-(\alpha_{m}-1) +\kappa( c( j-1,T) -c(
m,T)) }\right) ^{2}\right) .
\end{gather*}
Change the index of multiplication $i=j-1$ and this shows
\begin{gather}
\mathcal{E}( \widehat{\alpha},T) \frac{\Vert \zeta_{\alpha,T}\Vert ^{2}}{\Vert \zeta_{\widehat{\alpha}^{+},T}\Vert ^{2}}=( \alpha_{m}+\kappa c(m,T))\prod_{j=m+1}^{N}\left( 1-\left( \frac{\kappa}{\alpha_{m}+\kappa(c(m,T) -c(j,T)) }\right) ^{2}\right) .\label{normprod}
\end{gather}
This completes the proof.
\end{proof}

\section{Singular polynomials}

For special rational values of $\kappa$ there exist nonconstant polynomials $p$ in $\mathcal{P}_{\tau}$ which satisfy $\mathcal{D}_{i} p =0$ for $1 \leq i \leq N$. These are called singular polynomials and the corresponding value of~$\kappa$ is a~singular value. Suppose $\kappa$ is a~specific rational number for which the form $\langle\cdot, \cdot\rangle$ is positive-definite then $\kappa$ can not be a singular value, for suppose $p$ is singular then $\langle x^{\alpha} \otimes T ,p (x) \rangle_{\kappa} = \langle1 \otimes T$, $\mathcal{D}_{1}^{\alpha_{1}} \cdots\mathcal{D}_{N}^{\alpha_{N}} p (x) \rangle_{\kappa} =0$ for any $T \in\mathcal{Y} (\tau) $ and $\alpha\in\mathbb{N}_{0}^{N}$ with $\vert \alpha\vert >0$, so that $\langle p^{ \prime} (x) ,p (x) \rangle_{\kappa} =0$ for all $p^{ \prime} \in\mathcal{P} _{\tau}$ and $\Vert p\Vert ^{2} =0$. It is known that $\langle \cdot, \cdot\rangle_{\kappa}$ is positive-definite for $-\frac{1}{h_{\tau}}<\kappa<\frac{1}{h_{\tau}}$ (see formula~(\ref{JPnorm}) and~\cite{Dunkl2016}). We will construct singular polynomials for $\kappa= \pm\frac{1}{h_{\tau}}$ provided $\dim V_{\tau} \geq2$. Etingof and Stoica \cite[Section~5]{Etingof/Stoica2009} constructed singular polynomials for these parameter values without using Jack polynomials. In the one-dimensional cases $\tau=(N) $ the bound is $\kappa> -\frac{1}{h_{\tau}} =
-\frac{1}{N}$ and for $\tau=\big( 1^{N}\big) $ the bound is $\kappa <\frac{1}{h_{\tau}} =\frac{1}{N}$.

First suppose $\ell(\tau) \geq2$ then set $l =\ell ( \tau ) $, $\alpha=\big( 1^{\tau_{l}} ,0^{N -\tau_{l}}\big) $ and $T=T_{1}$, the inv-minimal RSYT which has $N ,N -1 ,\ldots,1$ entered row-by-row.

\begin{Theorem}\label{rowsthm}$\mathcal{D}_{\tau_{l}}\zeta_{\alpha,T_{1}}=\prod\limits_{i=1}^{l-1}\frac{1-\kappa h(i,1) }{1-\kappa(h(i,1) -1) }\theta_{\tau_{l}}^{-1}\zeta_{\widehat{\alpha},T_{1}}$.
\end{Theorem}

There are several ingredients to the proof. The first $\tau_{l}$ coordinates of the spectral vector of~$\zeta_{\alpha,T_{1}}$ are
\begin{gather*}
( 2 +\kappa( \tau_{l} -l) ,2 +\kappa( \tau_{l} -1 -l) ,\ldots,2 +\kappa( 1 -l)) .
\end{gather*}
The contents $c( r_{\alpha} (i) ,T_{1}) $ for $\tau_{l} +1 \leq i \leq N$ make up $l -1$ lists of consecutive integers, one for each row, from row $\# 1$ to row $\# (l -1) $. The following is easily proved by induction:

\begin{Lemma}Suppose $g$ is a function on $\mathbb{Z}$ and $a\leq b$ then
\begin{gather*}
\prod_{i=a}^{b}\left( \frac{g(i-1) g(i+1) }{g(i) ^{2}}\right) =\frac{g(a-1) g(b+1) }{g(a) g(b) }.
\end{gather*}
\end{Lemma}

Consider the part of the product in (\ref{normprod}) corresponding to row $\# i$: the contents are $1 -i ,2 -i ,\ldots,\tau_{i} -i$ and by the lemma this row contributes
\begin{gather*}
\frac{( 1 +\kappa( 1 -l -(1 -i) +1)) ( 1 +\kappa( 1 -l -( \tau_{i} -i) -1)) }{( 1 +\kappa( 1 -l -(1 -i) )) (1 +\kappa( 1 -l -( \tau_{i} -i)))}\\
\qquad {} =\frac{( 1 +\kappa( i -l +1) )( 1 -\kappa h ( i ,1) ) }{( 1 +\kappa( i -l)) ( 1 -\kappa( h ( i ,1) -1) ) }%
\end{gather*}
to the product because the hook-length $h ( i ,1) =\tau_{i} +l-i$. Thus
\begin{gather*}
\frac{\Vert \zeta_{\alpha,T_{1}}\Vert ^{2}}{\Vert \zeta_{\widehat{\alpha} ,T_{1}}\Vert ^{2}} =( 1 +\kappa(1 -l) ) \prod_{i =1}^{l -1}\frac{( 1 +\kappa( i -l+1) ) ( 1 -\kappa h ( i ,1) )}{( 1 +\kappa( i -l) ) ( 1 -\kappa( h( i ,1) -1) ) }\\
\hphantom{\frac{\Vert \zeta_{\alpha,T_{1}}\Vert ^{2}}{\Vert \zeta_{\widehat{\alpha} ,T_{1}}\Vert ^{2}}}{} =\prod_{i =1}^{l -1}\frac{1 -\kappa h ( i ,1) }{1 -\kappa( h ( i ,1) -1) },
\end{gather*}
the other factors telescope.

\begin{Lemma}\label{Dzi<m}Suppose $1\leq i\leq\tau_{l}$ then $\mathcal{D}_{i}\zeta _{\alpha,T_{1}}=\prod\limits_{i=1}^{l-1}\frac{1-\kappa h(i,1)
}{1-\kappa ( h(i,1) -1 ) }\theta_{i}^{-1}\zeta_{\widehat{\alpha},T_{1}}$.
\end{Lemma}

\begin{proof}By Case \ref{ai=ai1}(1) $s_{i} \zeta_{\alpha,T_{1}} =\zeta_{\alpha,T_{1}}$ for $1 \leq i \leq\tau_{l} -1$. Arguing inductively assume the stated formula (known true for $i =\tau_{l}$) and apply $s_{i -1}$ to both sides, then $s_{i-1} \mathcal{D}_{i} \zeta_{\alpha,T_{1}} =\mathcal{D}_{i -1} s_{i} \zeta_{\alpha,T_{1}}$ and $s_{i -1} \theta_{i}^{ -1} =s_{i -1} ( s_{i} \cdots s_{1} ) =\theta_{i -1}^{ -1}$.
\end{proof}

\begin{Theorem}\label{T1sing}$\zeta_{\alpha,T_{1}}$ is singular for $\kappa=\frac{1}{h_{\tau}}$.
\end{Theorem}

\begin{proof}There are no poles in any $\zeta_{\beta,T^{ \prime}}$ provided $-\frac{1}{h_{\tau} -1} <\kappa<\frac{1}{h_{\tau} -1}$ and the interval inclu\-des~$\frac{1}{h_{\tau}}$. The factor $1 -\kappa h ( 1 ,1 ) $ in the multiplicative constant shows that $\mathcal{D}_{i} \zeta_{\alpha,T_{1}} =0$ for $1 \leq i \leq\tau_{l}$. By Corollary~\ref{dz=zero} $\mathcal{D}_{i}\zeta_{\alpha,T_{1}} =0$ for $\tau_{l} <i \leq N$ because $\alpha_{j} =0$ for $j >\tau_{l}$.
\end{proof}

Each polynomial in $\operatorname{span} \{w\zeta_{\alpha,T_{1}} \colon w \in\mathcal{S}_{N}\} $ is also singular for $\kappa=1/h_{\tau}$. The NSJP's appearing in this way will be discussed in the sequel. For the $\kappa= -\frac{1}{h_{\tau}}$ case suppose $\tau_{1} \geq2$ and $\tau_{1} =\tau_{2} =\cdots=\tau_{m} >\tau_{m +1}$; then let $\alpha=\big( 1^{m} ,0^{N -m}\big) $ and $T =T_{0}$ the inv-maximal RSYT with $N ,N -1 ,\ldots,1$ entered column-by-column. Let $\tau^{ \prime}$ denote the transposed partition of $\tau$, then $\tau_{1}^{ \prime} =\ell(\tau)$ and $m =\tau_{\tau_{1}}^{ \prime}$.

\begin{Theorem}\label{colthm}$\mathcal{D}_{m}\zeta_{\alpha,T_{0}}=\prod\limits_{j=1}^{\tau_{1}-1}\frac{1+\kappa h( 1,j) }{1+\kappa( h(1,j) -1) }\theta_{\tau_{l}}^{-1}\zeta_{\widehat{\alpha},T_{0}}$.
\end{Theorem}

\begin{proof}The first $m$ coordinates of the spectral vector of $\zeta_{\alpha,T_{0}}$ are
\begin{gather*}
( 2 +\kappa( \tau_{1} -m) ,2 +\kappa( \tau_{1} -m+1) ,\ldots,2 +\kappa( \tau_{1} -1)) .
\end{gather*}
The contents $c ( r_{\alpha} (i) ,T_{1}) $ for $\tau_{l} +1 \leq i \leq N$ make up $\tau_{1} -1$ lists of consecutive integers, one for each column, from column $\# 1$ to column $\# (\tau_{1} -1) $. Consider the part of the product in~(\ref{normprod}) corresponding to column $\# j$: the contents are $j -\tau_{i}^{ \prime} ,j -\tau_{i}^{ \prime} +1 ,\ldots,j -1$ and by the Lemma this column contributes
\begin{gather*}
\frac{( 1 +\kappa( \tau_{1} -1 -( j -1) +1)) ( 1 +\kappa( \tau_{1} -1 -( j -\tau_{i}^{ \prime}) -1) ) }{1 +\kappa( \tau_{1} -1 -( j-1) ) ( 1 +\kappa( \tau_{1} -1 -( j -\tau_{i}^{ \prime}) ) ) }\\
\qquad{} =\frac{( 1 +\kappa( \tau_{1} -j +1) ) ( 1+\kappa h ( 1 ,j) ) }{( 1 +\kappa( \tau_{1}-j) ) ( 1 +\kappa( h ( 1 ,j) -1)) }
\end{gather*}
to the product because the hook-length $h ( 1 ,j) =\tau_{1}+\tau_{j}^{ \prime} -j$. Thus
\begin{gather*}
\frac{\Vert \zeta_{\alpha,T_{1}}\Vert ^{2}}{\Vert\zeta_{\widehat{\alpha} ,T_{1}}\Vert ^{2}} =( 1 +\kappa(\tau_{1} -1) ) \prod_{j =1}^{\tau_{1} -1}\frac{( 1+\kappa( \tau_{1} -j +1) ) ( 1 +\kappa h ( 1,j) ) }{( 1 +\kappa( \tau_{1} -j) )( 1 +\kappa( h ( 1 ,j) -1) ) }\\
\hphantom{\frac{\Vert \zeta_{\alpha,T_{1}}\Vert ^{2}}{\Vert\zeta_{\widehat{\alpha} ,T_{1}}\Vert ^{2}}}{} =\prod_{j =1}^{\tau_{1} -1}\frac{1 +\kappa h ( 1 ,j) }{1+\kappa( h ( 1 ,j) -1) },
\end{gather*}
the other factors telescope.
\end{proof}

\begin{Lemma}Suppose $1\leq i\leq m$ then $\mathcal{D}_{i}\zeta_{\alpha,T_{0}}=(-1) ^{m-i}\prod\limits_{j=1}^{\tau_{1}-1}\frac{1+\kappa h(
1,j) }{1+\kappa( h( 1,j) -1) }\theta_{i}^{-1}\zeta_{\widehat{\alpha},T_{1}}$.
\end{Lemma}

\begin{proof}By Case \ref{ai=ai1}(2) $s_{i} \zeta_{\alpha,T_{0}} = -\zeta_{\alpha,T_{0}}$ for $1 \leq i \leq m -1$. The rest of the argument is the same as in Lemma~\ref{Dzi<m}.
\end{proof}

The proof of the following is essentially the same as that of Theorem~\ref{T1sing}.

\begin{Theorem}$\zeta_{\alpha,T_{0}}$ is singular for $\kappa= -\frac{1}{h_{\tau}}$.
\end{Theorem}

As mentioned before the polynomials in $\operatorname{span} \{ w\zeta_{\alpha,T_{0}} \colon w \in\mathcal{S}_{N} \} $ are singular for the same $\kappa$.
\begin{Example} Let $\tau= ( 4 ,4 ,3 ,1 ) $ then $h_{\tau} =7$ and
\begin{gather*}
T_{0} =\left[
\begin{matrix}
12 & 8 & 5 & 2\\
11 & 7 & 4 & 1\\
10 & 6 & 3 & \\
9 & & &
\end{matrix} \right] ,\qquad T_{1} =\left[
\begin{matrix}
12 & 11 & 10 & 9\\
8 & 7 & 6 & 5\\
4 & 3 & 2 & \\
1 & & &
\end{matrix}
\right] \\
\alpha =\big( 1 ,1 ,0^{10}\big) ,\qquad \widehat{\alpha} =\big( 0 ,1,0^{10}\big) ,\beta=\big( 1 ,0^{11}\big) , \qquad \widehat{\beta} =\big(
0^{12}\big) ,\\
\mathcal{D}_{2} \zeta_{\alpha,T_{0}} =\frac{( 1 +7 \kappa)( 1 +5 \kappa) }{( 1 +6 \kappa) ( 1 +3\kappa) } s_{1} \zeta_{\widehat{\alpha} ,T_{0}} , \qquad \mathcal{D}_{1} \zeta_{\beta,T_{1}} =\frac{( 1 -7 \kappa) ( 1 -4\kappa) }{( 1 -5 \kappa) ( 1 -3 \kappa) }( 1 \otimes T_{1}) .
\end{gather*}
Thus $\zeta_{\alpha,T_{0}}$, $\zeta_{\beta,T_{1}}$ are singular for $\kappa= -\frac{1}{7}$, $\kappa=\frac{1}{7}$ respectively.
\end{Example}

The results of Section~\ref{ExtP} specialize to the symmetric groups $\mathcal{S}_{N}$, of type $A_{N -1}$, as follows: let $V:=\Big\{ x\in\mathbb{R}^{N} \colon \sum\limits_{i =1}^{N}x_{i} =0\Big\} $, the space on which the reflection representation acts irreducibly, then $\tau_{m}$ on $\wedge ^{m}(V) $ is isomorphic to the representation labeled by $\tau^{\prime}:=\big( N -m ,1^{m}\big) $ whose content sum $\varepsilon ( \tau^{ \prime}) =N \big( \frac{N -1}{2} -m\big) $. The singular polynomials for $\pm1/h_{t^{ \prime}} = \pm1/N$ are of degree one.

\subsection{The isotype of a space of singular polynomials}

The subspace of $\mathcal{P}_{\tau}$ of polynomials homogeneous of degree $n$ can be completely decomposed into subspaces irreducible and invariant under the action of $\mathcal{S}_{N}$, and these subspaces have bases of $ \{\omega_{i}\} $-simultaneous eigenvectors. Suppose $\lambda$ is a~partition of $N$ then a basis $\{ p_{S}\colon S\in\mathcal{Y} ( \lambda) \} $ (of an $\mathcal{S}_{N}$-invariant subspace) is called \textit{a basis of isotype} $\lambda$ if each $p_{S}$ transforms under the action of~$s_{i}$ according to Case~\ref{ai=ai1}. The key point here is when does a subspace have a basis of isotype~$\lambda$ made up of NSJP's. The transformation coefficients can be written as
\begin{gather*}
\frac{\kappa}{\xi_{\alpha,T}(i) -\xi_{\alpha,T} (i+1) }=\frac{1}{\big( \frac{\alpha_{i}}{\kappa}+c ( r_{\alpha}(i) ,T) \big) -\big( \frac{\alpha_{i+1}}{\kappa}+c( r_{\alpha}(i+1) ,T) \big) }.
\end{gather*}
So if $\zeta_{\alpha,T}=p_{S}$, that is, $\mathcal{U}_{i}\zeta_{\alpha,T}=( \alpha_{i}+1+\kappa c( r_{\alpha}(i),T)) \zeta_{\alpha,T}= (1+\kappa\omega_{i})p_{S}= (1+\kappa c( i,S)) p_{S}$ and the equivalence is $\frac{\alpha_{i}}{\kappa}+c ( r_{\alpha}(i) ,T ) =c ( i,S ) $ (note $T$ and $S$ have different shapes). (Some aspects of the argument are omitted; this discussion is meant as illustration). Consider $\zeta_{\alpha,T_{0}}$ (as in Theorem~\ref{colthm}) with $\kappa=-\frac{1}{h_{\tau}}$. Apply the transformation
\begin{gather}
\alpha_{i}+1+\kappa c ( r_{\alpha}(i) ,T )\mapsto\frac{\alpha_{i}}{\kappa}+c( r_{\alpha}(i),T) =c( i,S) \label{tfct}
\end{gather}
to $\xi_{\alpha,T}$. The result is (with $m=\tau_{\tau_{1}}^{\prime}$, the length of the last column of $\tau$)
\begin{gather*}
 \big({-}h_{\tau}+\tau_{1}-m,-h_{\tau}+\tau_{1}-m+1,\ldots,-h_{\tau}+\tau_{1}-1,[ c( i,T_{0}) ] _{i=m+1}^{N}\big) \\
\qquad{} =\big( 1-\tau_{1}^{\prime}-m,2-\tau_{1}^{\prime}-m,\ldots,-\tau_{1}^{\prime}, [ c ( i,T_{0} ) ] _{i=m+1}^{N}\big).
\end{gather*}
This is the content vector of the RSYT obtained from $T_{0}$ by removing the last column of $T_{0}$ and attaching it to the bottom of the first column. The $\mathcal{S}_{N}$-invariant subspace spanned by the orbit of $\zeta_{\alpha,T_{0}}$ is of isotype $\lambda$ where $\lambda^{\prime}= \big(\tau_{1}^{\prime}+\tau_{\tau_{1}}^{\prime},\tau_{2}^{\prime},\ldots,\tau_{\tau_{1}-1}^{\prime}\big) $ (another proof in \cite[Theorem~5.9]{Etingof/Stoica2009}).

The analogous computation for $\zeta_{\alpha,T_{1}}$ and $\kappa=\frac{1}{h_{\tau}}$ gives the content vector (recall $l =\ell(\tau)$)
\begin{gather*}
 \big( h_{\tau} -l +\tau_{l} ,h_{\tau} -l +\tau_{l} -1 ,\ldots,h_{\tau} -l+1 ,[ c( i ,T_{1})] _{i =\tau_{l} +1}^{N}\big)\\
\qquad{} =\big( \tau_{1} +\tau_{l} -1 ,\tau_{1} +\tau_{l} -2 ,\ldots,\tau_{1},[ c ( i ,T_{1})] _{i =\tau_{l} +1}^{N}\big).
\end{gather*}
This is the content vector of the RSYT obtained from $T_{1}$ by removing the last row of $T_{0}$ and attaching it to the right end of the first row. The $\mathcal{S}_{N}$-invariant subspace spanned by the orbit of $\zeta_{\alpha,T_{0}}$ is of isotype $\lambda= ( \tau_{1} +\tau_{l} ,\tau_{2} ,\ldots,\tau_{l -1} ) $. In the trivial cases $\tau=(N) $ one has $\kappa= -\frac{1}{N}$ with $\lambda=( N -1 ,1) $, and for $\tau=\big( 1^{N}\big) $ one has $\kappa=\frac{1}{N}$ and $\lambda=\big( 2 ,1^{N -1}\big) $.

By this argument and the definition of isotype for each RSYT $S$ of shape $\lambda$ there is a NSJP in $\operatorname{span} \{w\zeta_{\alpha,T_{1}}\colon w\in\mathcal{S}_{N}\} $ whose spectral vector with $\kappa=1/h_{\tau}$ transforms (by~(\ref{tfct})) to the content vector of~$S$. This polynomial is $\zeta_{\beta,T}$ where $\beta_{i}=1$ if $\operatorname{row} ( i,S ) =1$ and $\tau_{1}+1\leq \operatorname{col} ( i,S ) \leq\tau_{1}+\tau_{l}$ and $\beta _{i}=0$ otherwise; (for a tableau $T^{\prime}$ let $T^{\prime}[i,j] $ denote the entry at row $\#i$ and column $\#j$) then
\begin{gather*}
T [ 1,j] =S [ 1,j ] ,\qquad 1\leq j\leq\tau_{1},\\
T [ i,j ] =S [ i,j ] +\# \{ k\colon k>\tau _{1},\,S [ 1,k ] >S [ i,j ] \} ,\qquad 2\leq i<l, \qquad 1\leq j\leq\tau_{i},\\
T[ l,j] =\tau_{l}+1-j,\qquad 1\leq j\leq\tau_{l}.
\end{gather*}
An analogous formula holds for the $\zeta_{\alpha,T_{0}}$ and $\kappa=-1/h_{\tau}$ situation.

\begin{Example}
Here is an example using $\tau=( 3,3,3,3) $ and $N=12$, $\kappa=\frac{1}{6}$ and $\lambda=(6,3,3) $. Let
\begin{gather*}
S=\left[
\begin{matrix}
12 & 10 & 9 & 8 & 7 & 5\\
11 & 6 & 2 & & & \\
4 & 3 & 1 & & &
\end{matrix}
\right]
\end{gather*}
then $\beta=( 0,0,0,0,1,0,1,1,0,0,0,0) $ and
\begin{gather*}
T=\left[
\begin{matrix}
12 & 10 & 9\\
11 & 8 & 5\\
7 & 6 & 4\\
3 & 2 & 1
\end{matrix}
\right].
\end{gather*}
\end{Example}

\begin{Remark}We did not give any details about the indecomposable finite reflection groups with two conjugacy classes of reflections, the types $B_{N}$ and $F_{4}$. Griffeth developed the theory of vector-valued Jack polynomials for the complex reflection groups $G(n,p,N)$. The analogous result for $B_{N}$ of ours on type-$A$ singular Jack polynomials with minimal singular value can be derived from Griffeth's formulae~\cite{Griffeth2018} specialized to~$G(2,1,N)$, the hyperoctahedral group~$B_{N}$. This should show that the region of positivity of the Gaussian form is a neighborhood of the origin in $\mathbb{R}^{2}$ bounded by straight lines corresponding to singular values.
\end{Remark}

\pdfbookmark[1]{References}{ref}
\LastPageEnding

\end{document}